\long\def\symbolfootnote[#1]#2{\begingroup\def\thefootnote{\fnsymbol{footnote}}\footnote[#1]{#2}\endgroup}

\documentclass[a4paper,12pt]{amsart}
\usepackage{geometry}
\usepackage{hyperref}
\usepackage{indentfirst}
\usepackage{amssymb,amsmath}
\usepackage{mathrsfs}
\usepackage{amsfonts}
\usepackage[all]{xy}
\usepackage{tikz}
\usepackage{cancel}
\usetikzlibrary{calc,fadings,arrows.meta,decorations.pathreplacing}

\newtheorem{theorem}{Theorem}[section]
\newtheorem{corollary}[theorem]{Corollary}
\newtheorem{lemma}[theorem]{Lemma}
\theoremstyle{remark}
\newtheorem{remark}[theorem]{Remark}
\newtheorem{notation}[theorem]{Notation}
\theoremstyle{definition}
\newtheorem{definition}[theorem]{Definition}

\newtheorem{proposition}[theorem]{Proposition}

\newtheorem{assumption}[theorem]{Assumption}
\numberwithin{equation}{section}

\begin{document}
\author{Bo Liu, Mengqing Zhan}
\title[]{Eta Form and Spectral Sequence for the Composition of Fibrations}%\uppercase\expandafter{\romannumeral 1}.}
\date{}
\maketitle

\begin{abstract} 
In this paper, inspired by the spectral sequences constructed by signature operators with respect to the composition of fibrations, we define the "spectral sequences" for fiberwise Dirac operators and prove the equivariant family version of the adiabatic limit formula of eta invariants using the heat kernel method and the analytic localization techniques established by Bismut-Lebeau. In our formula, the
remainder terms are constructed by "spectral sequences" and completely extend those of Dai and Bunke-Ma.
\\

\end{abstract}

\maketitle



%\tableofcontents

\section{Introduction}\label{secIntroduction}

The Bismut-Cheeger eta form serves as the family extension of the eta invariant in index theory, which originally comes from the adiabatic limit of eta invariants. This limit is initiated by E. Witten \cite{Witten1985} for physical consideration and well studied by Bismut-Cheeger \cite{bismut1989eta} and Dai \cite{dai1991adiabatic}. In the general case of the 
adiabatic limit for Dirac operators in \cite{dai1991adiabatic}, a global
spectral term arises from the (asymptotically) very 
small eigenvalues. If we consider the signature 
operators, this spectral term can be constructed by 
Leray spectral sequences.

In \cite{bunke2004index}, in order to discuss the 
secondary index theory for flat bundles with duality, Bunke and Ma generalize the signature operators to the flat case and the adiabatic limit formula to the family case. In this case, the spectral terms are generalized to the finite dimensional eta forms constructed by spectral sequences.

In \cite{liu2017functoriality,liu2021equivariant,liu2021bismut}, for Dirac operators,
 the first author generalizes the adiabatic limit 
 formula to the equivariant family case for a fiberwise
 Lie group action. In \cite{liu2021bismut}, the spectral 
 terms are explained as equivariant Dai-Zhang higher spectral flows \cite{dai1998higher}. But those higher spectral 
 flow terms cannot degenerate to the terms in \cite{bunke2004index} and \cite{dai1991adiabatic} directly when restricted on the cases there.

In this paper, we make use of the descriptions in \cite{berthomieu1994quillen} to define a series of vector bundles over the base manifold which can be taken as the analogy of spectral sequences. Then the generalization of the remainder terms in \cite{bunke2004index} and \cite{dai1991adiabatic} 
are finite dimensional eta forms associated with these vector bundles. Moreover, these terms can also be considered as the refinement of the remainder terms in \cite{liu2021bismut}.

Now we explain our result in some details.

\iffalse Let $(X,g^{TX})$ be a closed oriented Riemannian manifold. 
Let $\underline{\mathcal{E}_X}:=(\mathcal{E}_X, h^{\mathcal{E}_X}, \nabla^{\mathcal{E}_X})$ be 
a $\mathbb{Z}_2$-graded self-adjoint Clifford 
module with Clifford connection (see \cite[Definition 3.32 and 3.39]{berline1992heat} for the definition). 
Let $D_X^{\mathcal{E}}$ be the Dirac operator twisted with these geometric 
data. 
%Here $\underline{X}$ denotes the manifold $X$ with metric $g^{TX}$. 
Let $\eta\left(D_X^{\mathcal{E}}\right)$ be
the Atiyah-Patodi-Singer eta invariant for this Dirac operator \cite{atiyah1973spectral}.
\fi

Let $\pi_X:W\to V$ be a submersion of two closed manifolds
with oriented closed fiber $X$.
Let $TX:=\ker (\pi_{{X},*}:TW\to TV)$ be the relative tangent bundle
over $W$.
%such that the relative tangent bundle $TX:=\ker (\pi_*:TW\to TB)$ over $W$ has a spin structure.
Let $T^HW$ be a horizontal subbundle of $TW$ such that $TW=T^HW\oplus TX$.
Let $g^{TX}$ be a metric on $TX$. 
Let $\underline{\mathcal{E}_X}=(\mathcal{E}_X, h^{\mathcal{E}_X}, \nabla^{\mathcal{E}_X})$ be 
a $\mathbb{Z}_2$-graded self-adjoint $\mathrm{Cl}(TX)$-Clifford 
module with Clifford connection (see (\ref{bl0959}) and (\ref{e01024})). Let $D_X^{\mathcal{E}_X}$ be the fiberwise
Dirac operators associated with $(g^{TX}, \nabla^{\mathcal{E}_X})$
(see (\ref{eq:2.11})).

%Let $\ul{E}$ be a geometric triple over $W$.
%Let $D_X^E$ be the fiberwise Dirac operator (see (\ref{eq:1.02})).
Assume that $\ker D_X^{\mathcal{E}_X}$ forms a vector bundle
over $V$. Under this assumption, the Bismut-Cheeger eta form
%Bismut and Cheeger defined a differential form 
$\tilde{\eta}
(\underline{\pi_X},\underline{\mathcal{E}_X})\in \Omega^*(V)$ (non-equivariant version of Definition \ref{defnetaform})
is well-defined.

Let $g^{TV}$ be a Riemannian metric on $TV$.
Let $\nabla^{TV}$ be the Levi-Civita connection.
%Assume that $V$ is spin and $TX$ is spin. Then $W$ is also spin.
Let $\underline{\mathcal{E}_V}=(\mathcal{E}_V, h^{\mathcal{E}_V},
\nabla^{\mathcal{E}_V})$ be a $\mathbb{Z}_2$-graded self-adjoint Clifford module over
$V$ with Clifford connection.
For $T>0$, let $g_T^{TW}:=\pi_X^*g^{TV}\oplus T^{-2}g^{TX}$, which is a Riemannian metric on $TW$.
%Let $\mathcal{S}(TV)$ be the spinor bundle over $V$ with Hermitian
%metric $h^{\mathcal{S}}$ and Clifford connection $\nabla^{\mathcal{S}}$.
Let $\mathcal{E}=\pi_X^*\mathcal{E}_V\widehat{\otimes}\mathcal{E}_X$.
Let $\nabla^{\mathcal{E},T}$ be the connection on $\mathcal{E}$
defined in (\ref{eq:3.26}).
Then $\underline{\mathcal{E}}=(\mathcal{E}, \pi_X^*h^{\mathcal{E}_V}\otimes h^{\mathcal{E}_X},\nabla^{\mathcal{E},T})  $
is a $\mathbb{Z}_2$-graded self-adjoint Clifford module over $W$ with Clifford connection associated with $g_T^{TW}$.
Let $D_{W,T}^{\mathcal{E}}$ and $D_V^{\mathcal{E}_V\otimes \ker D_X^{\mathcal{E}_X}}$ be the Dirac operator associated with $(g_T^{TW}, \nabla^{\mathcal{E},T})$ and $(g^{TV},\nabla^{\mathcal{E}_V}\otimes 1 + 1\otimes \nabla^{\ker D_X^{\mathcal{E}_X}})$ (see (\ref{eq:2.16}) for the definition of $\nabla^{\ker D_X^{\mathcal{E}_X}}$). 
%The following theorem, called the adiabatic limit formula,
%was established
%by Bismut and Cheeger in \cite{BC89} when $D_X^E$ is invertible
%and later extended to the case that $\ker D_X^E$ forms a vector bundle \cite{Dai91}. See also \cite[(54)]{Lo94} and \cite[Theorem 2.5]{Zh94} for some discussions and applications.
Let $\eta(D_{W,T}^{\mathcal{E}})$ and $\eta(D_V^{\mathcal{E}_V\otimes \ker D_X^{\mathcal{E}_X}})$ be the corresponding Atiyah-Patodi-Singer ets invariants in \cite{atiyah1973spectral}. The famous adiabatic limit formula is stated as follows.

\begin{theorem}\label{thm:0.03} \cite{bismut1989eta,dai1991adiabatic}
	%Assume that $V$ is spin and $TX$ is spin. Then $W$ is spin.
	If $\dim W$ is odd,
	\begin{align}\label{eq:0.05}
	\lim_{T\to +\infty}\eta\left(D_{W,T}^{\mathcal{E}} \right)
	=
%	\left\{
%	\begin{aligned}
	2\int_V \widehat{\mathrm{A}}(TV, \nabla^{TV})\tilde{\eta}
	(\underline{\pi_X},\underline{\mathcal{E}_X})+\eta\big(D_V^{\mathcal{E}_V\otimes\ker D_X^{\mathcal{E}_X}} \big) +R,
	\end{align}
	where $R$ is an integer-valued remainder term and  $\widehat{\mathrm{A}}(\cdot)$ is the corresponding $\widehat{\mathrm{A}}$-form (see \cite[\S 1.5]{berline1992heat} for the definition).
	Moreover, 
	\begin{enumerate}
		\item \cite[(0.5)]{bismut1989eta} if $W$ and $V$ are spin, and if 
		$D_X^{\mathcal{E}_X}$ is invertible, then $R=0$;
		\item \cite[Theorem 0.1]{dai1991adiabatic} if $W$ and $V$ are spin, 
		$\ker D_X^{\mathcal{E}_X}$ forms a vector bundle over $V$,
		and if $\dim \ker D_{W,T}^{\mathcal{E}}$ is independent of
		$T$,
		then 
		\begin{align}
		R=\sum_{\lambda\in A_r/A_{r+1}, r\geq 2} \mathrm{sgn}(\lambda),\quad 
		A_r := \left\{ \lambda\in\mathrm{Sp}(D_{W,T}^{\mathcal{E}}):\lambda=\mathrm{O}\left(\frac{1}{T^{r-1}}\right) \right\},
		\end{align}
		%which generalize the case above and
		where $\mathrm{Sp}(D_{W,T}^{\mathcal{E}})$ is the set of the spectrum of $D_{W,T}^{\mathcal{E}}$;
		\item \cite[Theorem 0.3]{dai1991adiabatic} if all Dirac operators are 
		signature operators, then $R$ is the sum of the signatures
		of the spectral sequences associated with the fiber bundle $\pi_X$.
	\end{enumerate}

\end{theorem}

Note that the case (1) in Theorem \ref{thm:0.03} is a special case of
case (2). And if $\dim V$ is even, then the term $\eta(D_V^{\mathcal{E}_V\otimes\ker D_X^{\mathcal{E}_X}} )$ in (\ref{eq:0.05}) vanishes. 

\begin{remark}
\begin{enumerate}
\item In \cite{bismut1989eta,dai1991adiabatic},
(\ref{eq:0.05}) was only proved for these three cases 
listed in Theorem \ref{thm:0.03}. But the formula 
(\ref{eq:0.05}) for Clifford modules is the natural extension of 
these results.
\item In \cite{bismut1989eta,dai1991adiabatic}, the authors
use the rescaling $g_t^{TW}=g^{TX}+t^{-2}\pi_X^*g^{TV}$, $t\to 0$.
Here we can consider $T=t^{-1}$. Then $g_T^{TW}=t^2g_t^{TW}$.
Note that when we multiply a constant on the metric, the 
eta invariant does not change. So (\ref{eq:0.05}) is the same
as the results in \cite{bismut1989eta,dai1991adiabatic}.
\item Let 
$\nabla_T^{TW}$ be the Levi-Civita connection associated with
$\pi_X^*g^{TV}\oplus T^{-2}g^{TX}$. Then by \cite[Proposition 4.5]{liu2017functoriality}
(cf. also \cite[(4.32)]{xiaonan1999formes}), $\displaystyle\lim_{T\to +\infty}\widetilde{\widehat{\mathrm{A}}}(TW,\nabla_T^{TW},\,^0\nabla^{TW})=0$, where $\,^0\nabla^{TW}=\pi_X^*\nabla^{TV}\oplus \nabla^{TX}$ and $\widetilde{\widehat{\mathrm{A}}}(\cdot)$ is the Chern-Simons form for the $\widehat{\mathrm{A}}$-form (cf. \cite[Definition B.5.3]{ma2007holomorphic}).  
The formula (\ref{eq:0.05}) can also be formulated as 
an equality:
	\begin{multline}\label{eq:1.03}
\eta\left(D_{W,T}^{\mathcal{E}} \right)
=
%	\left\{
%	\begin{aligned}
2\int_V \widehat{\mathrm{A}}(TV, \nabla^{TV})\tilde{\eta}
(\underline{\pi_X},\underline{\mathcal{E}_X})
-2\int_W\widetilde{\widehat{\mathrm{A}}}(TW,\nabla_T^{TW}, \,^0\nabla^{TW})
+\eta\big(D_V^{\mathcal{E}_V\otimes\ker D_X^{\mathcal{E}_X}} \big) +R.
\end{multline}
We usually take $T=1$.
\end{enumerate}
\end{remark}

Note that if $V$ is a point and $\dim X$ is odd, 
\begin{align}\label{eq:0.07}
\tilde{\eta}
(\underline{\pi_X},\underline{\mathcal{E}_X})=\frac{1}{2}\eta(D_X^{\mathcal{E}_X}).
\end{align}
Thus the Bismut-Cheeger eta form can be considered as the higher degree version of the eta invariant. If $V$ is a fibration
over a closed manifold $S$, then $W$ is also a fibration over $S$.
Then we could generalize the eta invariants in (\ref{eq:1.03})
to the Bismut-Cheeger eta forms. In fact, we could generalize them
directly to the equivariant eta forms for compact Lie group action.

Let $G$ be a compact Lie group. Let $W$, $V$, $S$ be closed $G$-manifolds. Let $\pi_X: W\rightarrow V$, $\pi_Y: V\rightarrow S$ be equivariant submersions with closed oriented
fibers $X$, $Y$. Then $\pi_Z=\pi_Y\circ \pi_X: W\rightarrow S$ is an equivariant submersion with closed oriented fiber $Z$.
Assume that $G$ acts on $S$ trivially.
We have the diagram of fibrations:
\begin{equation}\label{diagram}
\xymatrix{
	X\ar[r] & Z\ar[r]\ar[d]^{\pi_X} & W\ar[rd]^{\pi_Z}\ar[d]^{\pi_X} & \\ & Y\ar[r] & V\ar[r]^{\pi_Y} & S.
}
\end{equation}

Let $\underline{\pi_X}=(\pi_X, T_X^HW, g^{TX})$, $\underline{\pi_Y}=(\pi_Y, T_Y^HV, g^{TY})$
and $\underline{\pi_Z}=(\pi_Z, T_Z^HW, g^{TZ})$ be equivariant geometric data with respect to
$\pi_X$, $\pi_Y$ and $\pi_Z$ as in (\ref{eq:2.10}).
Assume that $T_Z^HW\subset T_X^HW$ and $g^{TZ}=\pi_X^*g^{TY}\oplus g^{TX}$.
Let $\nabla^{TX}$, $\nabla^{TY}$ and $\nabla^{TZ}$ be the
corresponding connections on $TX$, $TY$ and $TZ$ as in 
(\ref{connectiononrelativebundle}). Set $\,^0\nabla^{TZ}:=\pi_X^*\nabla^{TY}
\oplus \nabla^{TX}$.

Let $\underline{\mathcal{E}_X}=(\mathcal{E}_X, h^{\mathcal{E}_X},
\nabla^{\mathcal{E}_X})$ (resp.  $\underline{\mathcal{E}_Y}=(\mathcal{E}_Y, h^{\mathcal{E}_Y},
\nabla^{\mathcal{E}_Y})$) be a $\mathbb{Z}_2$-graded $G$-equivariant self-adjoint $\mathrm{Cl}(TX)$-module over $W$ (resp. $\mathrm{Cl}(TY)$-module
over $V$) with a $G$-invariant Clifford connection as in (\ref{eq:2.10}).
Let $\mathcal{E}=\pi_X^*\mathcal{E}_Y \widehat{\otimes} \mathcal{E}_X$.
Then $\underline{\mathcal{E}}=(\mathcal{E}, \pi_X^* h^{\mathcal{E}_Y}\otimes h^{\mathcal{E}_X}, \nabla^{\mathcal{E}})$ is a $\mathbb{Z}_2$-graded $G$-equivariant self-adjoint $\mathrm{Cl}(TZ)$-module over $W$ with a $G$-invariant Clifford connection. Here $\nabla^{\mathcal{E}}$
is defined in (\ref{eq:3.05}).

Let $D_X^{\mathcal{E}_X}$ and $D_Z^{\mathcal{E}}$ be fiberwise Dirac operators
associated with $(g^{TX}, \nabla^{\mathcal{E}_X})$ and $(g^{TZ}, \nabla^{\mathcal{E}})$ respectively. Assume that $\ker D_X^{\mathcal{E}_X}$ (resp. $\ker D_Z^{\mathcal{E}}$) forms 
a vector bundle over $V$ (resp. $S$). 
Let $\nabla^{\ker {D_X^{\mathcal{E}_X}}}$ be the 
induced $G$-invariant connection on the vector bundle $\ker D_X^{\mathcal{E}_X}$
as in (\ref{eq:2.16}).
Let $D_Y^{\mathcal{E}_Y\otimes\ker D_X^{\mathcal{E}_X}}$ be the fiberwise Dirac operator twisted with 
the vector bundle $\ker D_X^{\mathcal{E}_X}$ over $V$ associated with
$(g^{TY}, \nabla^{\ker {D_X^{\mathcal{E}_X}}})$. Assume that $\ker D_Y^{\mathcal{E}_Y\otimes\ker D_X^{\mathcal{E}_X}}$ forms 
a vector bundle over $S$.

\begin{theorem}\cite[Theorem 1.6]{liu2021bismut}\label{thm:0.06}
For $g\in G$, modulo exact forms on $S$, we have
\begin{equation}\label{main}
\begin{aligned}
\widetilde{\eta}_g & (\underline{\pi_Z}, \underline{\mathcal{E}})  = \widetilde{\eta}_g(\underline{\pi_Y},\underline{\mathcal{E}_Y\otimes\ker D_X^{\mathcal{E}_X}}) + \int_{Y^g} \widehat{\mathrm{A}}_g(TY,\nabla^{TY})\mathrm{ch}_g(\mathcal{E}_Y/\mathcal{S},\nabla^{\mathcal{E}_Y})\widetilde{\eta}_g(\underline{\pi_X^g},\underline{\mathcal{E}_X}) \\
& - \int_{Z^g}\widetilde{\widehat{\mathrm{A}}}_g(TZ,\nabla^{TZ},{^0\nabla^{TZ}})\mathrm{ch}_g(\mathcal{E}/\mathcal{S},\nabla^{\mathcal{E}})+\widetilde{R},
%& + \sum_{r=2}^\infty \widetilde{\eta}_g(\mathscr{E}_r,\mathscr{E}_{r+1},\nabla^r,\nabla^{r+1}) + \widetilde{\mathrm{ch}}_g(\ker D^Z,\nabla^\infty,\nabla^{\ker D^Z}).
\end{aligned}
\end{equation}
where $\underline{\pi_X^g}$ is defined in (\ref{eq:3.14}) and 
$\widetilde{R}\in \mathrm{ch}_g(K^0_G(S))$, the image of the equivariant Chern character $\mathrm{ch}_g$ on the equivariant topological $K$-group of $S$.
Here, $Y^g$ and $Z^g$ are the fixed point sets of $g\in G$ on $Y$ and $Z$ respectively, which are assumed to be oriented,  $\widehat{\mathrm{A}}_g(\cdot)$ and $\mathrm{ch}_g(\mathcal{E}_Y/\mathcal{S},\nabla^{\mathcal{E}_Y})$ are the equivariant $\widehat{\mathrm{A}}$-form
and the equivariant relative Chern character form (see, e.g., \cite[(1.32), (1.33)]{liuma2022}
 for the definitions) and $\widetilde{\widehat{\mathrm{A}}}_g(\cdot)$ is the equivariant
 Chen-Simons form associated with the equivariant $\widehat{\mathrm{A}}$-form which is the natural equivariant 
 extension of \cite[Definition B.5.3]{ma2007holomorphic}.
\end{theorem}

\begin{remark}\label{remark:1.04}
\begin{enumerate}
\item In \cite{liu2017functoriality}, the first author proves that
if there is no higher spectral flow for any deformation there, then
$\widetilde{R}=0$.
\item In \cite[Theorems 5.9 and 5.10]{bunke2004index}, if all Clifford modules are exterior algebra bundles twisted with flat bundles and the Dirac operators are generalized signature operators, 
%\begin{align}\label{eq:1.07}
%\widetilde{R}=\sum_{r=2}^\infty \widetilde{\eta}(\mathscr{E}_r,\mathscr{E}_{r+1},\nabla^r,\nabla^{r+1}) + \widetilde{\mathrm{ch}}(\ker D^Z,\nabla^\infty,\nabla^{\ker D^Z}).
%\end{align}
%The definitions of 
without the group action, Bunke and Ma show that the remainder term $\widetilde{R}$ is the sum of finite dimensional eta forms
constructed by spectral sequences.
If $S$ is a point and all flat bundles are trivial line bundles,
then $\widetilde{R}$ in this case degenerates to $R/2$ in 
Theorem \ref{thm:0.03} (3).
\item The proof of such formula is highly related to the analytical localization technique developed in \cite{bismut1991complex}.
\end{enumerate}
\end{remark}

The purpose of this paper is to establish the following result, which we state in
Theorem \ref{thmmaintheorem}.

\begin{theorem}
Under the setting of Theorem \ref{thm:0.06} and Assumption \ref{assumptions}, for $g\in G$, modulo exact forms on $S$,
(\ref{main}) holds and 
\begin{align}\label{eq:1.07}
\widetilde{R}=\sum_{r=2}^\infty \widetilde{\eta}_g(\mathscr{E}_r,\mathscr{E}_{r+1},\nabla^r,\nabla^{r+1}) + \widetilde{\mathrm{ch}}_g(\ker {D_Z^{\mathcal{E}}},\nabla^\infty,\nabla^{\ker {D_Z^{\mathcal{E}}}}).
\end{align}
The definitions of the notations above follow from (\ref{eq:3.22})  and (\ref{eq:3.34}).
\end{theorem}

Note that if $S$ is a point, the setting for the cases (1)-(3) in Theorem \ref{thm:0.03}
fulfills Assumption \ref{assumptions}. If $S$ is not a point,
without the group action,
the settings in Remark \ref{remark:1.04} also fulfills Assumption
\ref{assumptions}. In these cases, our theorem degenerates to previous results (see Proposition \ref{prop:5.04}).

\

\textbf{Notation}.
All manifolds in this paper are smooth and without boundary.
All fibrations in this paper are submersions with 
closed oriented fibers.
We denote by $d$ the exterior differential operator and 
$d^S$ when we like to insist the base manifold $S$.

We use the Einstein summation convention in this paper: when an index
variable appears twice in a single term and is not otherwise defined, it implies summation
of that term over all the values of the index.

We use the superconnection formalism of Quillen \cite{quillen1985superconnections}.
If $A$ is a $\mathbb{Z}_2$-graded algebra, and if $a,b\in A$, 
then we will note $[a,b]:=ab-(-1)^{\deg(a)\deg(b)}ba$ as the supercommutator 
of $a, b$.
If $E, E'$ are two $\mathbb{Z}_2$-graded spaces,
we will note $E\widehat{\otimes}E'$ as the $\mathbb{Z}_2$-graded 
tensor product as in \cite[\S 1.3]{berline1992heat}.
If one of $E, E'$ is ungraded, we understand 
it as $\mathbb{Z}_2$-graded by taking its odd part as zero. 

For the fiber bundle $\pi: W\rightarrow S$, 
we use the sign convention for the integration of  the differential forms along the oriented fibers $Z$ as follows:
for $\alpha\in \Omega^{\bullet}(S)$ and 
$\beta\in \Omega^{\bullet}(W)$, 
\begin{align}\label{e01136}
\int_Z(\pi^*\alpha)\wedge\beta=\alpha\wedge \int_Z\beta.
\end{align}

\section{Equivariant eta form}\label{s02}
In this section, we review the basic object of this paper --- eta forms. In Section 2.1, we describe the geometry of a fibration and introduce the Bismut superconnection to define the equivariant Bismut-Cheeger eta forms (cf. \cite{berline1992heat}). Then in Section 2.2, we introduce the finite version of eta forms for a vector bundle, which we call the equivariant finite dimensional eta forms. 

\subsection{Equivariant Bismut-Cheeger eta form}\label{s0201}

In this subsection, we recall the definition of the equivariant Bismut-Cheeger eta form. 

Given a submersion of closed manifolds $\pi:W\to S$ with closed oriented fiber $Z$, let $G$ be a compact Lie group which acts on $W$ with $\pi\circ g=\pi$, $\forall g\in G$.
In this case, the $G$-action on $S$ is trivial.
 We denote by $TZ:=\ker(\pi_*:TW\to TS)$ the relative tangent bundle and $T^HW$ a horizontal subbundle of $TW$ such that 
\begin{equation}\label{TWdecomposition}
  TW=T^HW\oplus TZ.
\end{equation}
Then $T^HW$ and $TZ$ are both vector bundles over $W$. We assume that the $G$-action preserves the orientation of $TZ$. We assume that $T^HW$ is also $G$-equivariant. 
%Such $T^HW$
%always exists due to the fact that $G$ is compact.
Then the $G$-action preserves the splitting (\ref{TWdecomposition}).
For $U\in TS$, let $U^H\in T^HW$ be its horizontal lift in $T^HW$ such that $\pi_*U^H=U$. 
 Let $P^{TZ}:TW\to TZ$ be the projection with respect to  (\ref{TWdecomposition}).

Let $g^{TZ}$ and $g^{TS}$ be $G$-invariant metrics on $TZ$ and $TS$ respectively. Then
\begin{equation}
  g^{TW} := \pi^*g^{TS} \oplus g^{TZ}
\end{equation}
is a $G$-invariant metric on $TW$.

Let $\nabla^{TW}$ be the Levi-Civita connection associated with $(TW,g^{TW})$ and
\begin{equation}\label{connectiononrelativebundle}
  \nabla^{TZ} := P^{TZ}\nabla^{TW}P^{TZ},
\end{equation}
which is a $G$-invariant Euclidean connection on $TZ$ depending only on $(T^HW,g^{TZ})$ (cf. \cite[Theorem 1.9]{bismut1986atiyah}). Let $\nabla^{TS}$ be the Levi-Civita connection on $(TS,g^{TS})$. Let 
\begin{equation}\label{0nablaTW}
  {^0\nabla^{TW}} := \pi^*\nabla^{TS} \oplus \nabla^{TZ}
\end{equation}
be a connection on $TW$, which is also $G$-invariant. We define 
\begin{equation}\label{eq:2.05}
	\mathcal{S}:=\nabla^{TW} - {^0\nabla^{TW}}.
\end{equation}
Then $\mathcal{S}$ is a 1-form on $W$ with values in antisymmetric elements of $\mathrm{End}(TW)$.
Let $\mathcal{T}$ be the torsion of $^0\nabla^{TW}$. 
Then by \cite[(1.30)]{bismut1986atiyah}, for $U,V\in TS$,
\begin{align}\label{eq:2.06}
\mathcal{T}(U^H, V^H)=-P^{TX}[U^H, V^H]\in TZ.
\end{align}

Let $\mathrm{Cl}(TZ)$ be the Clifford algebra bundle of $(TZ, g^{TZ})$, whose fiber at 
$x\in W$ is the Clifford algebra
$\mathrm{Cl}(T_xZ)$ of the Euclidean vector space $(T_xZ, g^{T_xZ})$. 
A $\mathbb{Z}_2$-graded self-adjoint $\mathrm{Cl}(TZ)$-module,
\begin{align}\label{bl0959} 
\mathcal{E}=\mathcal{E}_+\oplus\mathcal{E}_-,
\end{align}
is a $\mathbb{Z}_2$-graded complex vector bundle equipped with a Hermitian metric $h^{\mathcal{E}}$ 
preserving the splitting 
(\ref{bl0959}) and a 
fiberwise Clifford multiplication $c$ of $\mathrm{Cl}(TZ)$ such that the action $c$ 
restricted to $TZ$ is skew-adjoint on $(\mathcal{E},h^{\mathcal{E}})$ and anticommutes (resp. commutes) with the $\mathbb{Z}_2$-grading if the dimension of the fibres is even (resp. odd).
Locally, the Clifford module $\mathcal{E}$ could be written as
\begin{align}
\mathcal{E}=S(TZ)\widehat{\otimes}E,
\end{align}
where $S(TZ)$ is the spinor and $E=E_+\oplus E_-$ is a $\mathbb{Z}_2$-graded
complex vector bundle. In this case, if $\dim Z$ is even,
$S(TZ)=S_+(TZ)\oplus S_-(TZ)$ and 
$$\mathcal{E}_+=\left(S_+(TZ)\otimes E_+ \right)\oplus
\left(S_-(TZ)\otimes E_- \right),\quad
\mathcal{E}_-=\left(S_+(TZ)\otimes E_- \right)\oplus
\left(S_-(TZ)\otimes E_+ \right); $$
if $\dim Z$ is odd, 
$$\mathcal{E}_+=S(TZ)\otimes E_+,\quad \mathcal{E}_-=S(TZ)\otimes E_-.$$
%Let $\tau^{\mE}$ be the $\Z_2$-grading of $\mE$ which is $\pm 1$ on 
%$\mE_{\pm}$.
Let $\nabla^{\mathcal{E}}$ be a Clifford connection on $\mathcal{E}$ 
associated with $\nabla^{TZ}$, that is,
$\nabla^{\mathcal{E}}$  preserves $h^{\mathcal{E}}$ and the splitting (\ref{bl0959}) and 
for any $U\in TW$, $V\in \mathcal{C}^{\infty}(W,TZ)$,
\begin{align}\label{e01024}
\left[\nabla_U^{\mathcal{E}}, c(V)\right]=c\left(\nabla^{TZ}_UV\right).
\end{align}
We assume that the action of $G$  could be lifted on $\mathcal{E}$ such that it is 
compatible with the Clifford action and preserves the splitting 
(\ref{bl0959}).
We assume that $h^{\mathcal{E}}$ and $\nabla^{\mathcal{E}}$ are $G$-invariant.

\begin{notation}
  We denote by 
  \begin{equation}\label{eq:2.10}
    \underline{\pi} := (\pi,T^HW,g^{TZ}),\quad \underline{\mathcal{E}} := (\mathcal{E},h^{\mathcal{E}},\nabla^{\mathcal{E}})
  \end{equation}
  the corresponding geometric data of the fibration $\pi$ and the Clifford module $\mathcal{E}$ introduced above.
\end{notation}

On Clifford module $\mathcal{E}$, we define a family of Dirac operators over $S$:
\begin{equation}\label{eq:2.11}
  D_Z^{\mathcal{E}} := \sum_{i=1}^{\dim Z}c(e_i)\nabla^{\mathcal{E}}_{e_i},
\end{equation}
for $\{e_i\}_{i=1}^{\dim Z}$ a local orthonormal frame of $(TZ,g^{TZ})$.
This definition is independent of the choice of $\{e_i\}_{i=1}^{\dim Z}$.

Let $\mathscr{E}_b$ be the space of smooth sections of $\mathcal{E}$ over $Z_b$, $b\in S$, equipped with the $L^2$-inner product 
\begin{equation}\label{<>scrE}
  \langle \cdot,\cdot \rangle_{\mathscr{E}_b} := \int_{Z_b} h^{\mathcal{E}}(\cdot,\cdot) d  v_Z.
\end{equation}
As in \cite{bismut1986atiyah}, we take $(\mathscr{E},\langle\cdot,\cdot\rangle_{\mathscr{E}})$ as an infinite dimensional vector bundle over $S$. Let $\nabla^{\mathscr{E}}$\label{nablaEZ} be the $\langle\cdot,\cdot\rangle_{\mathscr{E}}$-preserving connection on $\mathscr{E}$ with respect to $\nabla^{\mathcal{E}}$ defined in \cite[(1.7)]{bismut1986analysis}. 

Let $\{f_p\}$ be a local frame of $TS$ and $\{f^p\}$ be its dual. By (\ref{eq:2.06}), we denote by
\begin{align}
c(\mathcal{T})=\frac{1}{2}c(\mathcal{T}(f^H_p,f^H_q))f^p\wedge f^q\wedge=-\frac{1}{2}c(P^{TX}[f^H_p,f^H_q])f^p\wedge f^q\wedge.
\end{align}  
Let $B$ be the Bismut superconnection defined by (cf. \cite[P.336]{berline1992heat})
\begin{equation}\label{rescaledBismutsuperconnection}
  B:= D_Z^{\mathcal{E}} + \nabla^{\mathscr{E}} - \frac{c(\mathcal{T})}{4},
\end{equation}
which only depends on geometric data $(T^HW, g^{TZ}, \nabla^{\mathcal{E}})$.
For $u>0$, we denote $\delta_u$
the operator on $\Lambda^i(T^*S)\widehat{\otimes}\mathscr{E}$ by multiplying differential
forms by $u^{i/2}$.
Then for $u>0$, we define the rescaled Bismut superconnection
\begin{align}
  B_u:=u\delta_{u}\circ B\circ\delta_{u}^{-1}= \sqrt{u} D_Z^{\mathcal{E}} + \nabla^{\mathscr{E}} - \frac{c(\mathcal{T})}{4\sqrt{u}}.
\end{align}

Under the conditions above, we see that $D_Z^{\mathcal{E}}$ commutes with the $G$-action. Then for any $b\in S$, $\ker D_{Z_b}^{\mathcal{E}}$ is a 
finite dimensional $G$-representation.
 We assume that $\{\ker D_{Z_b}^{\mathcal{E}}\}_{b\in S}$ forms a vector bundle over $S$. Then $\langle\cdot,\cdot\rangle_{\mathscr{E}}$ induces a $G$-invariant metric on $\ker D_Z^{\mathcal{E}}$. Let $P^{\ker D_Z^{\mathcal{E}}}: \mathscr{E}\to \ker D_Z^{\mathcal{E}}$ be the orthogonal projection with respect to $\langle\cdot,\cdot\rangle_{\mathscr{E}}$. We define 
\begin{equation}\label{eq:2.16}
  \nabla^{\ker D_Z^{\mathcal{E}}} := P^{\ker D_Z^{\mathcal{E}}}\circ\nabla^{\mathscr{E}}\circ P^{\ker D_Z^{\mathcal{E}}},
\end{equation}
which is a $G$-invariant Hermitian connection on $\ker D_Z^{\mathcal{E}}$.

For a trace class element $P\in\Lambda(T^*S)\widehat{\otimes}\mathrm{End}(\mathscr{E})$, we denote by $\mathrm{Tr}^{\mathrm{odd}/\mathrm{even}}[P]$ the part of $\mathrm{Tr}_s[P]$ which take values in odd or even forms. Set
\begin{equation}
  \widetilde{\mathrm{Tr}}[P] := \begin{cases}
    \mathrm{Tr}_s[P], & \text{ if } \dim Z \text{ is even; } \\
    \mathrm{Tr}^{\mathrm{odd}}[P], & \text{ if } \dim Z \text{ is odd.}
  \end{cases}
\end{equation}
Here $\mathrm{Tr}_s[P]$ denotes the supertrace of $P$ as in \cite[\S 1.3]{berline1992heat}.

For $\alpha\in\Omega^i(S)$, set 
\begin{equation}\label{psiS}
  \psi_S(\alpha) = \begin{cases}
    \left( \frac{1}{2\pi\sqrt{-1}} \right)^{\frac{i}{2}}\cdot\alpha, & i \text{ is even;} \\
    \frac{1}{\sqrt{\pi}}\left( \frac{1}{2\pi\sqrt{-1}} \right)^{\frac{i-1}{2}}\cdot\alpha, & i \text{ is odd,}
  \end{cases}
\end{equation}
and 
\begin{equation}\label{tildepsiS}
  \widetilde{\psi}_S(\alpha) = \begin{cases}
    \frac{1}{\sqrt{\pi}}\psi_S\alpha, & i \text{ is even;} \\
    \frac{1}{2\sqrt{-1}\sqrt{\pi}}\psi_S\alpha, & i \text{ is odd.}
  \end{cases}
\end{equation}
%Note that $\widetilde{\psi}_S$ can be seen as $\psi_{S\times\mathbb{R}}$. 

For $\beta\in\Omega^\bullet(B\times [0,1]_u)$, if we write $\beta= \beta_0 + d  u\wedge\beta_1$, with $\beta_0,\beta_1\in\Omega(T^*S)$, we set 
\begin{equation}
  [\beta]^{d  u} := \beta_1.
\end{equation}

For $g\in G$, let $W^g$ be the fixed point set of $g$-action on $W$. Then $\pi|_{W^g}: W^g\to S$ is a fiber bundle with fiber $Z^g$. We assume that $TZ^g$ is oriented.

\begin{definition}[{\cite[Definition 2.3]{liu2017functoriality}}]\label{defnetaform}
  %Assume that $\ker D^Z$ is locally constant on $S$. 
  For $g\in G$,  the \emph{equivariant Bismut-Cheeger eta form} $\widetilde{\eta}_g(\underline{\pi},\underline{\mathcal{E}})\in\Omega^\bullet(S)$ is defined by 
  \begin{equation}
    \begin{aligned}
      \widetilde{\eta}_g(\underline{\pi},\underline{\mathcal{E}}) & := - \int_0^{+\infty} \left\{\psi_S \widetilde{\mathrm{Tr}}\left[ g\exp\left( -\left( B_u + d u\wedge\frac{\partial}{\partial u}\right)^2 \right) \right] \right\}^{d u}  d u \\
      & = \begin{cases}
        \int_0^{+\infty} \widetilde{\psi}_S\mathrm{Tr}^{\mathrm{even}} \left[ g \frac{\partial B_u}{\partial u} \exp(-B^2_u)\right]d u\in \Omega^{\mathrm{even}}(B;\mathbb{C}), &\text{if} \dim Z\text{ is odd;} \\
        \int_0^{+\infty} \widetilde{\psi}_S\mathrm{Tr}_s \left[ g \frac{\partial B_u}{\partial u} \exp(-B^2_u)\right]d u\in \Omega^{\mathrm{odd}}(B;\mathbb{C}),&\text{if} \dim Z\text{ is even.}
      \end{cases}
    \end{aligned}
  \end{equation}
\end{definition}

%For simplicity, we always assume that $TZ^g$ is oriented.

\subsection{Finite dimensional eta form}
%Logically, the Bismut superconnection is the natural extension of Quillen's superconnection in \cite{quillen1985superconnections}. 
In this subsection, we introduce the definition of equivariant finite dimensional eta form.

Let $E\to S$ be a $G$-equivariant $\mathbb{Z}_2$-graded vector bundle with Hermitian metric $h^E$, which preserves the $\mathbb{Z}_2$-grading. We assume that the $G$-action on $S$ is trivial and $h^E$ is $G$-invariant. Let $\nabla^E$ be a $G$-invariant connection preserving $h^E$. Take $V\in \mathrm{End}(E)$ commuting with the $G$-action. 
We assume that $V$  either commutes or anticommutes with the $\mathbb{Z}_2$-grading.
Set
\begin{equation}
  E':= \ker V.
\end{equation}
We assume that $\dim\ker V$ is locally constant. Then $\ker V\to S$ is a $G$-equivariant $\mathbb{Z}_2$-graded vector bundle. We define the equivariant geometric data on $E'$ by the orthogonal projection $P^{\ker V}: E\to E'$ as 
\begin{equation}
  h^{E'} := P^{\ker V}\circ h^E\circ P^{\ker V},\quad \nabla^{E'} := P^{\ker V}\circ\nabla^E\circ P^{\ker V}.
\end{equation}
Then $\nabla^{E'}$ is a connection preserving $h^{E'}$.
%, which is also $G$-invariant. 

Based on Quillen's work \cite{quillen1985superconnections}, we define the superconnection as follows. 
%Note that in the following, 
%$V\in \mathrm{End}(E)$ is either odd or even. We need to define 
%it for these two cases separately.
\begin{definition}\label{defnQuillensuperconnection}
  Let $\sigma$ be a quantity super-commutes with $\Omega^\bullet(S)$. We define a \emph{superconnection}
  $L:\Omega^\bullet(S,E) \to \Omega^\bullet(S,E)$ by:
  \begin{equation}\label{Quillensuperconnection}
    L:=
    \begin{cases}
    \nabla^E + V, &\text{if}\ V \text{ anticommutes with the $\mathbb{Z}_2$-grading;} \\
    \nabla^E + \sigma V,&\text{if}\ V \text{ commutes with the $\mathbb{Z}_2$-grading.}
    \end{cases}
  \end{equation}
\end{definition}

\begin{remark}
  In the sequel, in order to simpify the notations, when $V$ commutes with the $\mathbb{Z}_2$-grading, we also usually omit $\sigma$ and regard $V$ as a quantity commutes with differential forms of even degree and anti-commutes with differential forms of odd degree.
\end{remark}

For $u>0$, set
\begin{equation}\label{rescaledsuperconnection}
  L_u := \sqrt{u}\delta_u\circ L\circ \delta_u^{-1} = \sqrt{u}V + \nabla^E.
\end{equation}

\begin{definition}\label{defnfiniteetafrom}
  For $g\in G$, we define the \emph{equivariant finite dimensional eta form} by
  \begin{equation}\label{finiteetaform}
    \begin{aligned}
      \widetilde{\eta}_g(E,E',\nabla^E,\nabla^{E'}) & := -\int_0^{+\infty} \left\{ \psi_S\mathrm{Tr}_s \left[ g\exp\left( - \left( L_u + d u\wedge \frac{\partial}{\partial u} \right)^2\right) \right] \right\}^{d  u} d  u \\
      & = \int_0^{+\infty}\widetilde{\psi}_S\mathrm{Tr}_s\left[ g\frac{\partial L_u}{\partial u}\exp(-L_u^2)\right] d u.
    \end{aligned}
  \end{equation}
\end{definition}
The legitimacy of the definition follows from the equivariant
version of \cite[Theorem 9.7]{berline1992heat}. 
Moreover, by the equivariant version of \cite[(9.2)]{berline1992heat},
\begin{equation}
  d \widetilde{\eta}_g(E,E',\nabla^E,\nabla^{E'}) = \mathrm{ch}_g(E',\nabla^{E'}) - \mathrm{ch}_g(E,\nabla^E).
\end{equation}

\section{Functoriality of eta forms}
We will present our main result in this section. In Section 3.1, we investigate the geometry of a composition of fibrations, then define the Dirac operators. In Section 3.2, we construct a series of bundles for the composition of fibrations by means of Dirac operators in Section 3.1, which is an analogy of \cite[(6.9)]{berthomieu1994quillen}. We also define the equivariant eta forms in this part. In Section 3.3, we rescale the Bismut superconnection and related operators to use the method of adiabatic limit. In Section 3.4 we state our main result, which can be seen as the relation of the equivariant Bismut-Cheeger and finite dimensional eta forms associated with the composition of fibrations under some assumptions.

\subsection{Composition of fibrations}
{We revisit the geometric setting in Section \ref{secIntroduction} to make the definition clear.} Let $W$, $V$, $S$ be closed $G$-manifolds. Let $\pi_X: W\rightarrow V$, $\pi_Y: V\rightarrow S$ be equivariant submersions with closed oriented
fibers $X$, $Y$. Then $\pi_Z=\pi_Y\circ \pi_X: W\rightarrow S$ is an equivariant submersion with closed oriented fiber $Z$.
Assume that $G$ acts on $S$ trivially.

We denote by $TX$, $TY$, $TZ$ the corresponding relative tangent bundles for $\pi_X$, $\pi_Y$, $\pi_Z$, and $T^H_XW$, $T^H_YV$, $T^H_ZW$ the horizontal $G$-equivariant subbundles respectively. For $U\in TS$, $U'\in TV$, we shall denote by ${U'}^H_X\in T^H_XW$, $U^H_Y\in T^H_YV$, $U^H_Z\in T^H_ZW$ the corresponding horizontal lifts of $U'$, $U$, $U$ such that $\pi_{X,*}({U'}^H_X)=U'$, $\pi_{Y,*}(U^H_Y) = U$, $\pi_{Z,*}(U^H_Z) = U$. We assume that $T^H_ZW\subset T^H_XW$. Let $T^HZ:=T^H_XW\cap TZ$. We have a splitting $TZ=T^HZ\oplus TX$ such that $T^HZ\simeq \pi^*_XTY$.

Let $g^{TX}$, $g^{TY}$ be two $G$-invariant Euclidean metrics on relative tangent bundles $TX$, $TY$ respectively. We define $g^{TZ}:=\pi_X^*g^{TY}\oplus g^{TX}$ on $TZ$, which is also $G$-invariant. Let $\nabla^{TX}$, $\nabla^{TY}$, $\nabla^{TZ}$ be $G$-invariant connections defined in (\ref{connectiononrelativebundle}) on $TX$, $TY$, $TZ$ respectively. Let ${ ^0\nabla^{TZ}}$ be the connection
\begin{equation}\label{eq:3.02}
  {^0\nabla^{TZ}} := \pi_X^*\nabla^{TY} \oplus\nabla^{TX}.
\end{equation}

In this paper, we write $\{g_\alpha\}$, $\{e_i\}$, $\{f_p\}$ the local orthonormal frames of $(TS,g^{TS})$, $(TX,g^{TX})$, $(TY,g^{TY})$  correspondingly, and $\{g^\alpha\}$, $\{e^i\}$, $\{f^p\}$ the dual frames.

Let $\underline{\mathcal{E}_X}=(\mathcal{E}_X, h^{\mathcal{E}_X},
\nabla^{\mathcal{E}_X})$ (resp.  $\underline{\mathcal{E}_Y}=(\mathcal{E}_Y, h^{\mathcal{E}_Y},
\nabla^{\mathcal{E}_Y})$) be a $\mathbb{Z}_2$-graded $G$-equivariant self-adjoint $\mathrm{Cl}(TX)$-module over $W$ (resp. $\mathrm{Cl}(TY)$-module
over $V$) with a $G$-invariant Clifford connection.
Then $\mathcal{E}_X$ is a $G$-equivariant vector bundle over $W$
and $\mathcal{E}_Y$ is a $G$-equivariant vector bundle over $V$.
%
%Let $\mathcal{E}_X,\mathcal{E}_Y$ be $G$-equivariant Clifford modules of $\mathrm{Cl}(TX),\mathrm{Cl}(TY)$ respectively, $h^{\mathcal{E}_X}, h^{\mathcal{E}_Y}$ be $G$-invariant Hermitian metrics on $\mathcal{E}_X,\mathcal{E}_Y$ respectively. 
Set
\begin{equation}
  \mathcal{E} : = \pi_X^*\mathcal{E}_Y\otimes\mathcal{E}_X.
\end{equation}
Then $\mathcal{E}$
is a $\mathbb{Z}_2$-graded $G$-equivariant self-adjoint Clifford module of $\mathrm{Cl}(TZ)\simeq \pi_X^*\mathrm{Cl}(TY)\widehat{\otimes}\mathrm{Cl}(TX)$ with Hermitian metric $h^\mathcal{E}:= \pi_X^*h^{\mathcal{E}_Y}\otimes h^{\mathcal{E}_X}$. For $U\in TY$, the Clifford action $c(U)$
on $\mathcal{E}_Y$ are lifted on $\pi_X^*\mathcal{E}_Y$ as 
$c(U_X^H)$.
Set
\begin{equation}
  {^0\nabla^{\mathcal{E}}} := \pi_X^*\nabla^{\mathcal{E}_Y}\otimes 1+1\otimes\nabla^{\mathcal{E}_X}.
\end{equation}
%Remark that
%${^0\nabla^{\mathcal{E}}}$ is not a Clifford connection. Hence we define $\nabla^{\mathcal{E}}$ as the $G$-invariant Hermitian Clifford connection on $(\mathcal{E},h^\mathcal{E})$.
From \cite[(4.3)]{liu2017functoriality},
\begin{align}\label{eq:3.05}
\nabla^{\mathcal{E}}:={^0\nabla^{\mathcal{E}}}+\frac{1}{2}\langle \mathcal{S}_{X}(\cdot) e_i,f_{p,X}^H\rangle c(e_i)c(f_{p})
+\frac{1}{4}\langle \mathcal{S}_{X}(\cdot)f_{p,X}^H,f_{q,X}^H\rangle c(f_{p})c(f_{q})
\end{align}
is a $G$-invariant Clifford connection  on $(\mathcal{E},h^\mathcal{E})$ associated with $\nabla^{TZ}$.
Here $\mathcal{S}_X$ is the tensor in (\ref{eq:2.05}) associated with 
$\pi_X$.

Let $D_X^{\mathcal{E}_X}$ and $D_Z^{\mathcal{E}}$ be the family Dirac operators with respect to $(g^{TX},\nabla^{\mathcal{E}_X})$ and $(g^{TZ},\nabla^{\mathcal{E}})$ respectively. Denote by $\mathscr{E}_X$ the infinite dimensional vector bundle over $V$
associated with $\mathcal{E}_X$. We shall denote by $\langle\cdot,\cdot\rangle_{\mathscr{E}_X}$ the $L^2$-inner product on $\mathscr{E}_X$ and $\nabla^{\mathscr{E}_X}$ the $G$-invariant $\langle\cdot,\cdot\rangle_{\mathscr{E}_X}$-preserving connection
as before. 
We define the inner product $\langle\cdot,\cdot\rangle_{\mathcal{E}_Y\otimes\mathscr{E}_X}$, based on $h^{\mathcal{E}_Y}$ and $\langle\cdot,\cdot\rangle_{\mathscr{E}_X}$. Set 
\begin{equation}
  \nabla^{\mathcal{E}_Y\otimes \mathscr{E}_X} := \nabla^{\mathcal{E}_Y}\otimes 1 + 1\otimes\nabla^{\mathscr{E}_X}
\end{equation}
on the bundle $\mathcal{E}_Y\otimes\mathscr{E}_X\to V$. 

We assume that $\ker D_X^{\mathcal{E}_X}$ forms a vector bundle over $V$. Let
\begin{equation}\label{PkerDX}
  P^{\ker {D_X^{\mathcal{E}_X}}} : \mathcal{E}_Y\otimes\mathscr{E}_X\to \mathcal{E}_Y\otimes\ker 
  D_X^{\mathcal{E}_X}
\end{equation}
be
the orthogonal projection with respect to $\langle\cdot,\cdot\rangle_{\mathcal{E}_Y\otimes\mathscr{E}_X}$. It is clear that $P^{\ker{D_X^{\mathcal{E}_X}}}$ induces metric, denoted by $\langle\cdot,\cdot\rangle_{\mathcal{E}_Y\otimes\ker {D_X^{\mathcal{E}_X}}}$, and connection, denoted by $\nabla^{\mathcal{E}_Y\otimes\ker {D_X^{\mathcal{E}_X}}}$, on $\mathcal{E}_Y\otimes\ker D_X^{\mathcal{E}_X}$. 
Note that all these data constructed on $\mathcal{E}_Y\otimes\ker D_X^{\mathcal{E}_X}$ are $G$-invariant. 

Set 
\begin{equation}
  D_H :=\sum_{p=1}^{\dim Y} c(f_p)\nabla^{\mathcal{E}_Y\otimes\mathscr{E}_X}_{f^H_{p}}.
\end{equation}
Let $D_Y^{\mathcal{E}_Y\otimes\ker {D_X^{\mathcal{E}_X}}}$ be the fiberwise Dirac operator associated with $(g^{TY},\nabla^{\mathcal{E}_Y\otimes\ker {D_X^{\mathcal{E}_X}}})$. Then it is clear that 
\begin{equation}\label{DYdefinition}
  D_Y^{\mathcal{E}_Y\otimes\ker {D_X^{\mathcal{E}_X}}} = P^{\ker {D_X^{\mathcal{E}_X}}}\circ D_H\circ P^{\ker {D_X^{\mathcal{E}_X}}}.
\end{equation}
By \cite[Theorem 10.19]{berline1992heat}, we have 
\begin{equation}\label{DZdecomposition}
  D_Z^{\mathcal{E}} = D_X^{\mathcal{E}_X} + D_H + C,
\end{equation}
where
\begin{equation}\label{defnC}
  C = -\frac{1}{8} \langle \mathcal{T}_X(f^H_{p,X},f^H_{q,X}),e_i\rangle c(e_i)c(f_{p})c(f_{q}).
\end{equation}
Here $\mathcal{T}_X$ is the torsion tensor associated with the fibration $\pi_X$.

\subsection{The bundles of spectral sequence}
With all these geometric data prepared, in this subsection, we will define a series of bundles over $S$, denoted by $\{\mathscr{E}_r\}_{r=0,1,\cdots,\infty}$. 

Compared with \cite[(6.9)]{berthomieu1994quillen} and \cite[(2.13)]{ma2002functoriality}, we make the following definition.
\begin{definition}\label{defnEr}
  For $b\in S$, $r\in \mathbb{Z}_+$, we define
\begin{equation}\label{definitionofEr}
	\begin{aligned}
		\mathscr{E}_{r,b} := \Big\{s_0\in \mathscr{E}_b: & \text{ There exist } s_1,\cdots,s_{r-1}\in \mathscr{E}_b, \text{ such that } \\
    & D_X^{\mathcal{E}_X}s_0 = 0,\ D_Hs_0+ D_X^{\mathcal{E}_X}s_1 = 0,\ Cs_0+D_Hs_1 + D_X^{\mathcal{E}_X}s_2=0, \\
		&  \cdots, Cs_{r-3}+D_Hs_{r-2} + D_X^{\mathcal{E}_X}s_{r-1}= 0 \Big\}.
	\end{aligned}
\end{equation}
%where $C$ is defined by (\ref{defnC}). 
\end{definition}
In what follows, for $s_0\in\mathscr{E}_r$ we shall write $\varphi_r(s_0)=(s_1,\cdots,s_{r-1})$, which are the elements in (\ref{definitionofEr}).

Note that $\mathscr{E}_{1,b}=C^{\infty}(Y_b, \mathcal{E}_Y\otimes\ker D_X^{\mathcal{E}_X})$, hence $\mathscr{E}_{2,b}=\ker D_{Y_b}^{\mathcal{E}_Y\otimes\ker D_X}$ is a finite dimensional vector space. For $r> 2$, $\mathscr{E}_{r,b}\subset \mathscr{E}_{2,b}$. 
So for $r\geq 2$, $\dim\mathscr{E}_{r,b}<+\infty$. We assume that $\{\mathscr{E}_{r,b}\}_{b\in S}$ constitute complex vector bundles for $r\geqslant 2$. Then they should be $G$-equivariant and $\mathbb{Z}_2$-graded.

\begin{remark}
  When the Dirac operator is the Dolbeault operator, $\mathscr{E}_r$ reduces to \cite[(6.9)]{berthomieu1994quillen} and for signature operator, $\mathscr{E}_r$ becomes the term $\mathcal{E}_r$ in \cite[(2.13)]{ma2002functoriality} and \cite[Theorems 5.9, 5.10]{bunke2004index}.
  For Dolbeault and signature operator, $\mathscr{E}_r$ can be interpreted as terms of Leray spectral sequences (see \cite[Theorem 6.1]{berthomieu1994quillen} and \cite[Proposition 2.1]{ma2002functoriality}). However, in our general case, there is no topological meaning for $\mathscr{E}_r, r\geqslant 2$. Hence we need the assumption above.
\end{remark}

%We shall assume that $\ker D_X^{\mathcal{E}_X},\ker D^Z$ are both vector bundles,  which are naturally $G$-equivariant and $\mathbb{Z}_2$-graded. 

Now we will construct the geometric data and the equivariant eta form for $\mathscr{E}_r$. 
%\begin{itemize}
  
  For $r=0$, set $\mathscr{E}_0=\mathscr{E}$, the infinite dimensional bundle over $S$ whose fiber is the space of the 
  smooth sections of $\mathcal{E}$ over $Z$. By abusing the notation, we write $h_0$ for $\langle\cdot,\cdot\rangle_{\mathscr{E}}$, the metric defined in (\ref{<>scrE}) on $\mathscr{E}$. We write $\nabla^0=\nabla^{\mathscr{E}}$ and $D_0= D_X^{\mathcal{E}_X}$. Let $B_0$ be the Bismut superconnection associated with $(T^H_XW,g^{TX},\nabla^{\mathcal{E}_X})$. 
  
  For any $g\in G$, the equivariant Bismut-Cheeger eta form 
  \begin{equation}
    \widetilde{\eta}_g(\underline{\pi_X^g},\underline{\mathcal{E}_X})\in \Omega^\bullet(V^g),
  \end{equation}
  is well-defined as in (\ref{defnetaform}).  Here $\underline{\pi_X^g}$ stands for 
  \begin{equation}\label{eq:3.14}
    \underline{\pi_X^g} := \left(\pi_X|_{W^g},T^H_X(W|_{V^g}):=T^H_XW|_{V^g}\cap T(W|_{V^g}),g^{TX}\right).
  \end{equation}
  
  For $r=1$, $\mathscr{E}_1=\ker D_0=\ker D_X^{\mathcal{E}_X}$. Set $p_1=P^{\ker {D_X^{\mathcal{E}_X}}}: \mathscr{E}_0\to \ker D_0$.
  % be the orthogonal projection with respect to $h_0$, which is the same as $P^{\ker D_X^{\mathcal{E}_X}}$ in (\ref{PkerDX}).
  Let $h_1$ be the metric on $\mathscr{E}_1$ induced from $h_0$.
   Let
  \begin{equation}
    \nabla^1 := p_1\circ\nabla^0\circ p_1
  \end{equation}
  be the connection on $\mathscr{E}_1$ preserving $h_1$. We denote by $p_1^\bot := 1-p_1$. 

  Let $D_1:= D_Y^{\mathcal{E}_Y\otimes \ker {D_X^{\mathcal{E}_X}}}$ be the Dirac operator defined in (\ref{DYdefinition}). Let $B_1$ be the Bismut superconnection associated with $(T^H_YV,g^{TY},\nabla^{\mathcal{E}_Y\otimes\ker {D_X^{\mathcal{E}_X}}})$. For $g\in G$, we have
  \begin{equation}
    \widetilde{\eta}_g(\underline{\pi_Y},\underline{\mathcal{E}_Y\otimes\ker D_X^{\mathcal{E}_X}})\in\Omega^\bullet(S).
  \end{equation}
  
  For $r\geqslant 2$, let 
  \begin{equation}
    p_r:\mathscr{E}_0\to\mathscr{E}_r
  \end{equation}
  be the orthogonal projection with respect to $h_0$ and $p_r^\bot := 1-p_r$. Let  $h_r$ be the metric on $\mathscr{E}_r$ induced from $h_0$ and
  \begin{equation}
   \nabla^r:= p_r\circ\nabla^0\circ p_r.
  \end{equation}
  %Let $\mathscr{E}^\bot_{r+1}$ be the orthogonal complement of $\mathscr{E}_{r+1}$ with respect to $h_r$. 
  Then $\nabla^r$ preserves $h_r$.

  Comparing with \cite[(6.10)]{berthomieu1994quillen} and \cite[(2.14)]{ma2002functoriality}, we define $D_r$ on $\mathscr{E}_r$ by 
  \begin{equation}\label{eq:3.18}
    D_r s_0 := p_r(D_Hs_{r-1}+Cs_{r-2}),
  \end{equation}
  where $s_{r-1},s_{r-2}\in\mathscr{E}$ are the elements in (\ref{definitionofEr}). Then $D_r$ commutes (resp. anticommutes)
  with the $\mathbb{Z}_2$-grading on $\mathcal{E}_r$ when 
  $\dim Z$ is odd (resp. even).

%  Following the same process by downward recursion in the proof of \cite[Theorem 6.1]{berthomieu1994quillen}, we can prove the lemma below 
  
%\end{itemize}
\begin{lemma}\label{lemEr=kerDr-1}
  The operator $D_r$ in (\ref{eq:3.18}) is well-defined. That is, it is independent of the choice of $s_1,\cdots,s_{r-1}$ .
 \end{lemma}
\begin{proof}
	%\begin{itemize}
	When $r=1$, $D_1=D_X^{\mathcal{E}_X}$ is well-defined. We assume that $D_{r'}$ is well-defined for any $r'\leqslant r$. We shall prove that $D_{r+1}$ is legitimate. Suppose that $\varphi_r(s_0) = (s_1,\cdots,s_{r-1}), \varphi'_r(s_0) = (s_1',\cdots,s_{r-1}')$. We will show that $D_rs_0= p_r(D_Hs_{r-1}+Cs_{r-2}), D_r's_0= p_r(D_Hs'_{r-1}+Cs'_{r-2})$ define the same operator. 

  We claim that the following isomorphism holds for $r$:
  \begin{equation}\label{eq:3.19}
    \ker D_{r-1}\simeq \mathscr{E}_{r}.
  \end{equation}

  Then we define $t_k:=s_{k+1}-s'_{k+1}, k=0,\cdots,r-2$. So we have that $D_X^{\mathcal{E}_X}t_0 = 0$, $D_Ht_0 + D^{\mathcal{E}_X}_X t_1 = 0$, $\cdots$, $Ct_{r-4}+D_Ht_{r-3}+D_X^{\mathcal{E}_X}t_{r-2}=0$, which means that $t_1,\cdots,t_{r-2}$ make $t_0$ in $\mathscr{E}_{r-1}$. We have 
  \begin{equation}
    (D_r-D_r')s_0 = p_r(D_Ht_{r-2}+Ct_{r-3}) = p_rp_{r-1}(D_Ht_{r-2}+Ct_{r-3}) = p_r D_{r-1}t_0.
  \end{equation}
  By (\ref{eq:3.19}), $D_{r-1}t_0 \in \mathrm{im} D_{r-1} = (\mathscr{E}_r)^\bot$. Thus $(D_r-D_r')s_0=p_r D_{r-1}t_0=0 $. So $D_r=D_r'$, which proves the Lemma \ref{lemEr=kerDr-1}.

  Now we prove (\ref{eq:3.19}). For $r=1$, $\mathscr{E}_1\simeq \ker D_0$ follows from the definition. We assume that $\ker D_{r'-1}\cong \mathscr{E}_{r'}$ for  $r'\leqslant k, k<r$. We only need to prove that it holds for $k+1$.

		On one hand, for $s_0\in \mathscr{E}_{k+1}$, Let $\varphi_{k+1}(s_0) = (s_1,\cdots,s_k)$. By  (\ref{eq:3.18}),
		\begin{equation}
		D_ks_0 = p_k(D_Hs_{k-1}+Cs_{k-2}) = -p_kD_X^{\mathcal{E}_X}s_k.
		\end{equation}
		Since  $\mathscr{E}_1=\ker D_X^{\mathcal{E}_X}$, we know that $D_X^{\mathcal{E}_X}s_k\in (\ker D_X^{\mathcal{E}_X})^\bot = (\mathscr{E}_1)^\bot \subset (\mathscr{E}_k)^\bot$. So $D_ks_0=0$, which implies that $\mathscr{E}_{k+1}\subset \ker D_k$.

      On the other hand, we need to show that
      \begin{equation}\label{kerDrEr+1}
        \ker D_k\subset\mathscr{E}_{k+1}.
      \end{equation}
      Suppose that $s_0\in\mathscr{E}_k$ satisfying 
      \begin{equation}\label{Drs0=0}
        D_ks_0 = p_k(D_Hs_{k-1}+Cs_{k-2}) = 0.
      \end{equation}
    When $k=1$, by setting 
      \begin{equation}
          s_1= - (D_X^{\mathcal{E}_X})^{-1}p_1^\bot D_Hs_0,
      \end{equation}
      we fulfill the equation below 
      \begin{equation}
        D_X^{\mathcal{E}_X}s_1 + D_Hs_0=0.
      \end{equation}
      Now by assuming that (\ref{eq:3.19}) holds for any $k'\leqslant k$ and the assumption (\ref{Drs0=0}), we may write 
      \begin{equation}\label{kerDrEr+1'}
        \begin{aligned}
          D_Hs_{k-1}+Cs_{k-2} & = D_{k-1}s_0^{(1)} + p_{k-1}^\bot (D_Hs_{k-1}+Cs_{k-2}) \\
          & = D_{k-1}s_0^{(1)} + D_{k-2}s_0^{(2)} + p_{k-2}^\bot (D_Hs_{k-1}+Cs_{k-2}) \\
          & = \cdots = D_{k-1}s_0^{(1)} + D_{k-2}s_0^{(2)} + \cdots +D^{\mathcal{E}_X}_X s_0^{(k)},
        \end{aligned}
      \end{equation}
      where $s_0^{(i)}\in\mathscr{E}_{k-i}$ with $\varphi_{k-i}(s_0^{(i)}) = (s_1^{(i)},\cdots,s_{k-i-1}^{(i)})$. 

      Note that under the assumption (\ref{eq:3.19}), 
      \begin{equation}\label{kerDrEr+1''}
        \begin{aligned}
          D_{k-1}s_0^{(1)} & = p_{k-1}( D_H s_{k-2}^{(1)} + C s_{k-3}^{(1)})  = D_H s_{k-2}^{(1)} + C s_{k-3}^{(1)} - p_{k-1}^\bot ( D_H s_{k-2}^{(1)} + C s_{k-3}^{(1)}) \\
          & =  D_H s_{k-2}^{(1)} + C s_{k-3}^{(1)} + D_{k-2}s_{0}^{(1,1)} + p_{k-3}^\bot (D_H s_{k-2}^{(1)} + C s_{k-3}^{(1)}) \\
          & = \cdots = D_H s_{k-2}^{(1)} + C s_{k-3}^{(1)} + D_{k-2}s_{0}^{(2)'} + \cdots +D_X^{\mathscr{E}_X} s_0^{(k-1)'},
        \end{aligned}
      \end{equation}
      which means that by taking $\varphi_k(s_0)=(s_1,s_1-s_1^{(1)},\cdots,s_{k-1}-s_{k-1}^{(1)})=:(s_1',\cdots,s_{k-1}')$, we may write 
      \begin{equation}
        D_Hs_{k-1}'+Cs_{k-2}' =  D_{k-2}s_{0}^{(1,1)} + \cdots +D_X^{\mathscr{E}_X} s_0^{(1,k-1)}.
      \end{equation}
      At the end, we may find $\varphi_r(s_0)=(\widetilde{s}_1,\cdots,\widetilde{s}_{k-1})$, which makes 
      \begin{equation}
        D_H\widetilde{s}_{k-1}+C\widetilde{s}_{k-2}= -D_X^{\mathscr{E}_X} \widetilde{s}_k,
      \end{equation}
      for some $\widetilde{s}_k\in \mathscr{E}$. So that $\varphi_{k+1}(s_0)=(\widetilde{s}_1,\cdots,\widetilde{s}_k)$ is well-defined and $s_0\in\mathscr{E}_{k+1}$.
\end{proof}

\begin{corollary}
  For $r\geqslant 0$, we have 
  \begin{equation}
    \ker D_r \simeq \mathscr{E}_{r+1}.
  \end{equation}
\end{corollary}

  By assumption, for $r\geq 2$, $\mathscr{E}_r$ is a $G$-equivariant $\mathbb{Z}_2$-graded vector bundle over $S$. By Definition \ref{defnQuillensuperconnection}, we define the superconnection 
  \begin{equation}\label{definitionofBr}
    B_r := \nabla^r + D_r,\quad \text{ for } r\geqslant 2.
  \end{equation}
  By Lemma \ref{lemEr=kerDr-1} and Definition \ref{defnfiniteetafrom}, for any $g\in G$, we could define 
  \begin{equation}\label{eq:3.22}
    \widetilde{\eta}_g(\mathscr{E}_r,\mathscr{E}_{r+1},\nabla^r,\nabla^{r+1})\in\Omega^\bullet(S),\quad r\geqslant 2.
  \end{equation}
  
  As $\mathscr{E}_r\supset\mathscr{E}_{r+1}$ and $\dim\mathscr{E}_r<+\infty$ for $r\geq 2$, there exists $r_0$ such that for $r\geqslant r_0$, $\mathscr{E}_r\simeq\mathscr{E}_{r_0}$. We denote by $\mathscr{E}_\infty$ the convergent one.

  We assume that $\ker D_Z^{\mathcal{E}}$ forms a vector bundle over $S$.
  Let 
  \begin{equation}
    p_\infty: \mathscr{E}_0\to\mathscr{E}_\infty,\quad p:\mathscr{E}_0\to\ker D_Z^{\mathcal{E}}  
  \end{equation}
  be the orthogonal projections associated with $h_0$. We have the natural connection on $\ker D_Z^{\mathcal{E} }$:
  \begin{equation}\label{nablainftykerDZ}
    \nabla^{\ker {D_Z^{\mathcal{E}}}} := p\circ\nabla^0\circ p.
  \end{equation}

\subsection{The family of adiabatic limit}

In this subsection,  we will study a family of adiabatic limits over the base manifold $S$. 

 We define the $G$-invariant metrics over $TZ$ and $TW$ for $T\geq 1$:
\begin{equation}
  %\begin{aligned}
    g^{TZ}_T := \pi_X^*g^{TY} \oplus \frac{1}{T^2}g^{TX}, \quad   g^{TW}_T := \pi_Z^*g^{TS} \oplus g^{TZ}_T.
  %\end{aligned}
\end{equation}
Let $\mathrm{Cl}_T(TZ)$ be the Clifford algebra bundle associated
with $g_T^{TZ}$, the Clifford multiplication of which is denoted by
$c_T$.
%By the isomorphism of Clifford algebra, defined by
It is easy to see that the map
 $(\mathrm{Cl}_T(TZ),g_T^{TZ})\to (\mathrm{Cl}(TZ),g^{TZ})$,
 defined by $c_T(f_p)\mapsto c(f_p), c_T(Te_i)\mapsto c(e_i)$,
is an isomorphism of Clifford algebras. Then we could regard $\mathcal{E}$ as a Clifford module of $\mathrm{Cl}_T(TZ)$ through this isomorphism. Let $\nabla^{TZ,T}$ be the connection defined by (\ref{connectiononrelativebundle}) associated with $(T_Z^HW, g^{TZ}_T)$. Then from \cite[(4.3)]{liu2017functoriality}, as in (\ref{eq:3.05}), 
\begin{align}\label{eq:3.26}
\nabla^{\mathcal{E},T}:={^0\nabla^{\mathcal{E}}}+\frac{1}{2T}\langle \mathcal{S}_{X}(\cdot) e_i,f_{p,X}^H\rangle c(e_i)c(f_{p})
+\frac{1}{4T^2}\langle \mathcal{S}_{X}(\cdot)f_{p,X}^H,f_{q,X}^H\rangle c(f_{p})c(f_{q})
\end{align}
is a $G$-invariant Clifford connection associated with $\nabla^{TZ,T}$.
%
% The Clifford connection on $\mathcal{E}$ compatible with $\nabla^{TZ,T}$ through $c_T$ is written as $\nabla^{\mathcal{E},T}$.

Let $B_T$ be the Bismut
superconnection associated with $(T_Z^HW, g_T^{TZ}, \nabla^{\mathcal{E},T})$.
%
%Then we get many operators with parameter $T$, especially, the Dirac operator $D^Z_T$ and $B_T$. We have the following decomposition of these two operators. which can be found in \cite{liu2017functoriality}. 
%We shall assume that $\ker D^Z_T$ is a $G$-equivariant vector bundle over $S\times [1,+\infty]_T$, with which we can prove the following result:
%Let $\delta_u$ be the rescaling operator on $\Omega^\bullet(S)$, we define 
Let
\begin{equation}
	B_{u^2,T} := u\delta_{u^2}\circ B_T\circ\delta_{u^2}^{-1}.
\end{equation}
Let ${^0\nabla^{\mathscr{E}}}$ be the connection on $\mathscr{E}$  with respect to $^0\nabla^{\mathcal{E}}$ defined in \cite[(1.7)]{bismut1986analysis}, which preserves $\langle\cdot,\cdot\rangle_{\mathscr{E}}$ in (\ref{<>scrE}). 
\begin{theorem}\cite[Proposition 5.5]{liu2017functoriality}\label{thmBTBuTdecomposition}
	For $T>0, u>0$,
	\begin{multline}\label{eq:3.29}
			B_{u^2,T} =  uTD_X^{\mathcal{E}_X} + {^0\nabla^{\mathscr{E}}} + uD_H - \frac{1}{4u}\langle \mathcal{S}_X(f^H_{p,X})g^H_{\alpha,Z}, g^H_{\beta,Z} \rangle c(f_{p})g^{\alpha}\wedge g^{\beta}\wedge \\
			- \frac{u}{4T}\langle \mathcal{S}_X(e_i)f^H_{p,X}, f^H_{q,X} \rangle c(e_i)c(f_{p})c(f_{q}) 
			- \frac{1}{4uT}\langle \mathcal{S}_Z(e_i)g^H_{\alpha,Z},g^H_{\beta,Z}\rangle c(e_i)g^{\alpha}\wedge g^{\beta}\wedge \\
			 - \frac{1}{2T}\langle \mathcal{S}_X(e_i)f^H_{p,X},g^H_{\alpha,Z}\rangle c(e_i)c(f_{p})g^{\alpha}\wedge.
	\end{multline}
Here $\mathcal{S}_Z$ is the tensor in (\ref{eq:2.05}) associated with 
$\pi_Z$.
\end{theorem}
%	\begin{proof}
%	The proof follows the same as \cite[Proposition 5.5]{liu2017functoriality}.
%\end{proof}

Let ${D_{Z,T}^{\mathcal{E}}}$ be the fiberwise Dirac operator associated with
$(g_T^{TZ}, \nabla^{\mathcal{E},T})$. 
By taking the $0$-form part on $S$ in (\ref{eq:3.29}), 
we have
%\begin{theorem}
	\begin{equation}
	%    \begin{aligned}\label{DZTdecomposition}
	%      B_T & = D_T^Z + \nabla^{\mathscr{E},T} - \frac{c(T_{Z,T})}{4}, \\
	D_{Z,T}^{\mathcal{E}} = TD_X^{\mathcal{E}_X} + D_H + \frac{1}{T}C.
	%\end{aligned}
	\end{equation}
	%  Here $\nabla^{\mathscr{E},T}$ is the connection on $\mathscr{E}$ induced by $\nabla^{\mathcal{E},T}$. 
%	\begin{proof}
%		This follows easily from (\ref{DZdecomposition}).
%	\end{proof}
%\end{theorem}
Let ${\nabla^{\mathscr{E},T}}$ be the connection on $\mathscr{E}$  with respect to $\nabla^{\mathcal{E},T}$ defined in \cite[(1.7)]{bismut1986analysis}, which preserves $\langle\cdot,\cdot\rangle_{\mathscr{E}}$ in (\ref{<>scrE}).
By taking the $1$-form part on $S$ in (\ref{eq:3.29}), we see that
\begin{equation}\label{eq:3.30}
\nabla^{\mathscr{E},T} = {^0\nabla^{\mathscr{E}}} -\frac{1}{2T} \langle \mathcal{S}_X(e_i)f^H_{p,X},g^H_{\alpha,Z} \rangle c(e_i)c(f_{p})g^{\alpha}\wedge.
\end{equation}  

The following lemma will be proved in  Corollary \ref{corkerDZEinfty}.

\begin{lemma}
	We assume that $\ker D^Z_T$ is a $G$-equivariant vector bundle over $S\times [1,+\infty)_T$. Let $p^T:\mathscr{E}_0\to \ker D_{Z,T}^{\mathcal{E}}$ be the orthogonal projection associated with $h_0$. Then there exists $C>0$, $T_0\geq 1$, such that for any 
	$T\geq T_0$, $s\in \mathscr{E}_0$,
	\begin{align}
	\| p^Ts-p_{\infty}s\|\leq \frac{C}{T}\|s\|.
	\end{align}
	Therefore,
	we have 
	\begin{equation}
	\mathscr{E}_\infty \simeq \ker D_Z^{\mathcal{E}}.
	\end{equation}

\end{lemma}

%Let $p_T:\mathscr{E}\to \ker D^Z_T$ be the orthogonal projection with respect to $h_0$. 

Compared with (\ref{nablainftykerDZ}), we write 
\begin{equation}\label{eq:3.33}
  \nabla^{\ker {D_{Z,T}^{\mathcal{E}}}} := p^T\circ \nabla^{\mathscr{E},T}\circ p^T,
\end{equation}
From (\ref{eq:3.30}), $\lim_{T\to+\infty}\nabla^{\ker {D_{Z,T}^{\mathcal{E}}}}$ exists. We denote the limit by
\begin{equation}\label{nablainfinity}
  \nabla^\infty = p_\infty {^0\nabla^{\mathscr{E}}}p_\infty.
\end{equation}

We rearrange the parameter by $s=T^{-1}$. Then $\ker {D_{Z,1/s}^{\mathcal{E}}}$ is a $G$-equivariant vector bundle over $S\times [0,1]_s$.
%denoted 
%by $\ker \widetilde{D}^Z$. Let $\nabla^{\ker \widetilde{D}^Z}$ be the connection on $\ker \widetilde{D}^Z$ such that $\nabla^{\ker \widetilde{D}^Z}|_{S\times \{s\}}=\nabla^{\ker D_{1/s}^Z}$.
%As
Then we can define the equivariant Chern-Simons forms 
$\widetilde{\mathrm{ch}}_g(\ker {D_{Z}^{\mathcal{E}}}, \nabla^{\infty}, \nabla^{\ker {D_{Z}^{\mathcal{E}}}})$ as in \cite[(1.29)]{liumaInvent}. Moreover (see e.g., \cite[(1.30)]{liumaInvent}), 
\begin{align}\label{eq:3.34}
d\, \widetilde{\mathrm{ch}}_g(\ker D^Z, \nabla^{\infty}, \nabla^{\ker {D_{Z}^{\mathcal{E}}}})=\mathrm{ch}_g(\ker {D_{Z}^{\mathcal{E}}}, \nabla^{\ker {D_{Z}^{\mathcal{E}}}})
-\mathrm{ch}_g(\mathscr{E}_{\infty}, \nabla^{\infty}).
\end{align}

\subsection{The main result}
\begin{assumption}\label{assumptions}
For the main result of this paper, we make the following assumptions:
\begin{itemize}
  \item $ T^H_ZW \subset T^H_XW,\quad g^{TZ} = g^{TX}\oplus \pi_X^*g^{TY} $;
  \item $\ker D_X^{\mathcal{E}_X}$ is a vector bundle over $V$;
  \item $\mathscr{E}_r,r\geqslant 2$ are vector bundles over $S$;
  \item $\ker {D_{Z,T}^{\mathcal{E}}}$ is a vector bundle over $S\times [1,+\infty)_T$.
  \item For $g\in G$, $TX^g$ and $TY^g$ are oriented. So $TZ^g$
  is also oriented.
\end{itemize}
\end{assumption}

\begin{theorem}\label{thmmaintheorem}
  For $g\in G$, under the assumptions above, modulo exact forms on $S$, we have
  \begin{multline}\label{maintheorem}
		%\begin{aligned}
			\widetilde{\eta}_g  (\underline{\pi_Z}, \underline{\mathcal{E}})  = \widetilde{\eta}_g(\underline{\pi_Y},\underline{\mathcal{E}_Y\otimes\ker D_X^{\mathcal{E}_X}}) + \int_{Y^g} \widehat{\mathrm{A}}_g(TY,\nabla^{TY})\mathrm{ch}_g(\mathcal{E}_Y/\mathcal{S},\nabla^{\mathcal{E}_Y})\widetilde{\eta}_g(\underline{\pi_X^g},\underline{\mathcal{E}_X}) \\
			- \int_{Z^g}\widetilde{\widehat{\mathrm{A}}}_g(TZ,\nabla^{TZ},{^0\nabla^{TZ}})\mathrm{ch}_g(\mathcal{E}/\mathcal{S},\nabla^{\mathcal{E}})\\
			+ \sum_{r=2}^\infty \widetilde{\eta}_g(\mathscr{E}_r,\mathscr{E}_{r+1},\nabla^r,\nabla^{r+1}) + \widetilde{\mathrm{ch}}_g(\ker {D_{Z}^{\mathcal{E}}},\nabla^\infty,\nabla^{\ker {D_{Z}^{\mathcal{E}}}}).
		%\end{aligned}
	\end{multline}
\end{theorem}

\section{\textbf{The proof of Theorem \ref{thmmaintheorem}}}
In this section, we will prove our main result Theorem \ref{thmmaintheorem}. In Section 4.1, we define a second layer of adiabatic limit, then obtain a $1$-form on $\mathbb{R}\times\mathbb{R}$. In Section 4.2, we follow the same strategy as in \cite[\S 4]{liu2017functoriality} combined with the author's work in \cite{ma2000formes} to prove the main theorem. There is one intermediate theorem we need to prove, which is left to the next section.

\subsection{The fundamental form}
Let $\widehat{S}:=\mathbb{R}_{+,T}\times\mathbb{R}_{+,u}\times S$ and fibration $\widehat{\pi}_Z : \widehat{W} \to \widehat{S}$ with fiber $Z$. We define the metric $\widehat{g}^{TZ}$ of $\widehat{\pi}_Z$ such that 
\begin{equation}
  \widehat{g}^{TZ}|_{(T,u)} = u^{-2}g^{TZ}_T.
\end{equation}
Let $\widehat{P}_W:\widehat{W}\to W$ be the natural projection and $\widehat{\mathcal{E}}:=\widehat{P}_W^*\mathcal{E}$.
Let $\nabla^{\widehat{\mathcal{E}}}$ be the connection on $\widehat{\mathcal{E}}$ such that $\nabla^{\widehat{\mathcal{E}}}|_{(T,u)}
=\nabla^{\mathcal{E},T}$.
Then the Bismut superconnection $\widehat{B}$ with respect to $(\widehat{\pi}_Z,\widehat{g}^{TZ}, \nabla^{\widehat{\mathcal{E}}})$ can be written as (cf. \cite[(4.4)]{liu2017functoriality})
\begin{equation}
  \widehat{B}|_{(T,u)} = B_{u^2,T} + d  T\wedge\frac{\partial}{\partial T} + d  u\wedge\frac{\partial}{\partial u} - \frac{n}{2u}d  u - \frac{n-m}{2T}d  T.
\end{equation}

\begin{definition}
  Let $\gamma:=d  u \wedge \gamma^u + d T\wedge\gamma^T$ be the part of $\psi_{\widehat{S}}\widetilde{\mathrm{Tr}}[g\exp(-{\widehat{B}}^2)]$ of degree one with respect to the coordinate $(T,u)$ with functions $\gamma^u,\gamma^T : \mathbb{R}_{+,T}\times\mathbb{R}_{+,u} \to\Omega^\bullet(S)$.
\end{definition}

It follows from \cite[Proposition 4.2]{liu2017functoriality} that there exists a smooth family $\alpha: \mathbb{R}_{+,T}\times\mathbb{R}_{+,u} \to\Omega^\bullet(S)$ such that 
\begin{equation}\label{exactalpha}
  \left(d u\wedge\frac{\partial}{\partial u} + d T\wedge\frac{\partial}{\partial T}\right)\gamma = d T \wedge d u \wedge d ^S\alpha.
\end{equation}

We take $\varepsilon,A,T_0\in\mathbb{R}$ such that $0\leqslant \varepsilon\leqslant A <+\infty, 1\leqslant T_0<+\infty$. Set $\Gamma = \Gamma_{\varepsilon,A,T_0}$, the contour in $\mathbb{R}_{+,T}\times\mathbb{R}_{+,u}$ with four parts $\Gamma_1$, $\Gamma_2$, $\Gamma_3$, $\Gamma_4$ and $\mathcal{U}$ the domain enclosed by $\Gamma$, as in the Figure \ref{contour1}.
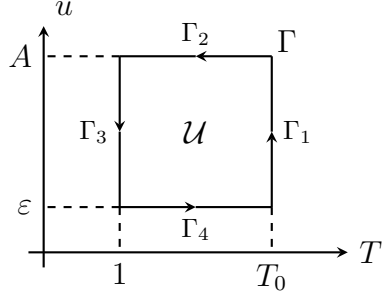
\begin{figure}[!htp]\centering\begin{tikzpicture}[>=stealth,thick]
  \coordinate (a) at (1,0. 6);
  \coordinate (b) at (3,0. 6);
  \coordinate (c) at (1,2. 6);
  \coordinate (d) at (3,2. 6);

  \draw[->] (-0. 2,0) --  (4,0) node[right] {$T$};
  \draw[->] (0,-0. 2) -- (0,3) node[above right] {$u$};

  \draw[->] (a) -- (2,0. 6) node[below] {\footnotesize $\Gamma_4$};
  \draw (2,0. 6) -- (b); 
  \draw[->] (b) -- (3,1. 6) node[right] {\footnotesize $\Gamma_1$};
  \draw (3,1. 6) -- (d); 
  \draw[->] (d) -- (2,2. 6) node[above] {\footnotesize $\Gamma_2$};
  \draw (2,2. 6) -- (c); 
  \draw[->] (c) -- (1,1. 6) node[left] {\footnotesize $\Gamma_3$};
  \draw (1,1. 6) -- (a); 

  \draw[dashed] (a) -- (1,0) node[below] {$1$};
  \draw[dashed] (b) -- (3,0) node[below] {$T_0$};
  \draw[dashed] (a) -- (0,0. 6) node[left] {$\varepsilon$};
  \draw[dashed] (c) -- (0,2. 6) node[left] {$A$};

\node at (2,1. 6) {$\mathcal{U}$};
\node at (3. 2,2. 8) {$\Gamma$};
\end{tikzpicture} 
\caption{The contour $\Gamma$}\label{contour1}
\end{figure}

We denote by $I^0_i:= \int_{\Gamma_i}\gamma$. Then by (\ref{exactalpha}) and Stokes formula, 
\begin{equation}
  \sum_{i=1}^4 I_i^0 = \int_{\partial\mathcal{U}}\gamma = d ^S \left(\int_{\mathcal{U}}\alpha d T\wedge d u\right).
\end{equation}
We take the limits $A\to+\infty, T\to+\infty$ and then $\varepsilon\to 0$ in the indicated order. Let $I_i^k$, $1\leqslant i\leqslant 4$, $ 1\leqslant k\leqslant 3$ denote the value of the part $I^0_i$ after taking the $k$-th limit. Then by \cite[\S 22, Theorem 17]{de1973varietes},
\begin{equation}\label{sumofI3i}
  \sum_{i=1}^4 I^3_i \equiv 0 \mod{d ^S\Omega^\bullet(S)}.
\end{equation}

\subsection{Intermediate results}
With all these superconnections $B_r$, $r\geqslant 0$, we can build the bridge between the fundamental form $\gamma$ and equivariant eta forms $\widetilde{\eta}_g$. 

%\begin{itemize}
  For $r=0$, we consider the fibration $\widetilde{W}|_{V^g}:=\mathbb{R}_{+,t}\times W|_{V^g}\to \widetilde{V}^g:=\mathbb{R}_{+,t}\times V^g$ with fiber $X$. 
  Let $\widetilde{P}_W:\widetilde{W}|_{V^g}\to W|_{V^g}$ be the natural projection.
  Set $T^H_X(\widetilde{W}|_{V^g}) = T(\mathbb{R}_+)\oplus\widetilde{P}_W^*(T^H_XW|_{V^g})$, $\widetilde{g}^{TX}|_{\{t\}} = t^{-2}g^{TX}$.
  Set $\widetilde{\mathcal{E}}_X:=\widetilde{P}_W^*\mathcal{E}_X$.
  Let $\nabla^{\widetilde{\mathcal{E}}_X}$ be the connection on $\widetilde{\mathcal{E}}_X$ such that $\nabla^{\widetilde{\mathcal{E}}_X}|_{\{t\}}
  =\nabla^{\mathcal{E}_X}$.
 Let $\widetilde{B}_0$ be the Bismut superconnection associated with $(T^H_X(\widetilde{W}|_{V^g}),\widetilde{g}^{TX}, \nabla^{\widetilde{\mathcal{E}}_X})$. 
 %Here we denoted by $\nabla^{\widetilde{\mathcal{E}_X}}$ the induced connection on $\mathrm{Pr}_W^*\mathcal{E}_X$. 
 We decompose 
  \begin{equation}
    \psi_{\widetilde{V}^g}\widetilde{\mathrm{Tr}}[g\exp(-\widetilde{B}_0^2)] = d t\wedge\gamma_0(t) + r_0(t),
  \end{equation}
  where $\gamma_0(t),r_0(t)\in\Omega^\bullet(V^g)$. By \cite[(4.18)]{liu2017functoriality}, we have 
  \begin{equation}
    \int_0^{+\infty}\gamma_0(t)d  t= - \widetilde{\eta}_g(\underline{\pi_X^g},\underline{\mathcal{E}_X}).
  \end{equation}
  
  For $r=1$, 
  we consider the fibration $\widetilde{V}:=\mathbb{R}_{+,t}\times V\to \widetilde{S}:=\mathbb{R}_{+,t}\times S$ with fiber $Y$. 
  Let $\widetilde{P}_V:\widetilde{V}\to V$ be the natural projection.
  Set $T^H_Y\widetilde{V} = T(\mathbb{R}_+)\oplus\widetilde{P}_V^*(T^H_XV)$, $\widetilde{g}^{TY}|_{\{t\}} = t^{-2}g^{TY}$.
  Set $\widetilde{\mathcal{E}}_Y:=\widetilde{P}_V^*\mathcal{E}_Y$.
  Let $\nabla^{\widetilde{\mathcal{E}}_Y}$ be the connection on $\widetilde{\mathcal{E}}_Y$ such that $\nabla^{\widetilde{\mathcal{E}}_Y}|_{\{t\}}
  =\nabla^{\mathcal{E}_Y}$.
%  we consider the projection $\widetilde{P}_V: \widetilde{V}=\to V$ and the fibration $\widetilde{V}\to \widetilde{S}$. For $t\in\mathbb{R}_+$ and $v\in V$, we take the geometric data $T^H_Y\widetilde{V} := T(\mathbb{R}_+)\oplus \mathrm{Pr}_V^*(T^H_YV)$, $\widetilde{g}^{TY}|_{(t,v)} = t^{-2}g^{TY}_v$. 
  Under the Assumption \ref{assumptions}, $\ker D_X^{\widetilde{\mathcal{E}}_X}$ is a vector bundle over $\widetilde{V}$. Let $h^{\ker {D_X^{\widetilde{\mathcal{E}}_X}}}$ and $\nabla^{\ker {D_X^{\widetilde{\mathcal{E}}_X}}}$ be the corresponding induced metric and connection.
  % We denoted by $\nabla^{\widetilde{\mathcal{E}_Y}}$ the induced connection on $\mathrm{Pr}_V^*\mathcal{E}_Y$.
  Let $\widetilde{B}_1$ be the Bismut superconnection with respect to 
  %\begin{equation}
   % \begin{aligned}
    $  (T^H_Y\widetilde{V},\widetilde{g}^{TY},\nabla^{\widetilde{\mathcal{E}}_Y\otimes\ker {D_X^{\widetilde{\mathcal{E}}_X}}})$.
%    \\
%      (T^H_YV,u^{-2}g^{TY},\nabla^{\mathcal{E}_Y\otimes\ker D_X^{\mathcal{E}_X}}).
%    \end{aligned}
 % \end{equation}
  We decompose 
  \begin{equation}
    \psi_{\widetilde{S}}\widetilde{\mathrm{Tr}}[g\exp(-\widetilde{B}_1^2)] = d t\wedge\gamma_1(t) + r_1(t),
  \end{equation}
  with $\gamma_1(t),r_1(t)\in\Omega^\bullet(S)$. Then by \cite[(4.12)]{liu2017functoriality},
  \begin{equation}
    \int_0^{+\infty}\gamma_1(t)d t = -\widetilde{\eta}_g(\underline{\pi_Y},\underline{\mathcal{E}_Y\otimes\ker D_X^{\mathcal{E}_X}}).
  \end{equation}
  
  For $r\geqslant 2$, we denote by $\widetilde{\mathscr{E}}_r\to\widetilde{S}$ the lift of $\mathscr{E}_r$ on $\widetilde{S}$. Let $\widetilde{D}_r$ be the operator on $\widetilde{\mathscr{E}}_r$ such that $\widetilde{D}_r|_{\{t\}\times S} = tD_r$. As in (\ref{Quillensuperconnection}), we define the superconnection $\widetilde{B}_r$. We decompose 
  \begin{equation}
    \psi_{\widetilde{S}}\mathrm{Tr}_s[g\exp(-\widetilde{B}^2_r)] = d t\wedge\gamma_r(t) + r_r(t),
  \end{equation}
  with $\gamma_r(t),r_r(t)\in\Omega^\bullet(S)$. By Definition \ref{defnfiniteetafrom}, we have 
  \begin{equation}\label{eq:4.11}
    \int_0^{+\infty}\gamma_r(t)d  t = -\widetilde{\eta}_g(\mathscr{E}_r,\mathscr{E}_{r+1},\nabla^r,\nabla^{r+1}).
  \end{equation}
%\end{itemize}

%Considering the Chern-Simons form of equivariant $\widehat{\mathrm{A}}$-form. 
Since $\widehat{\mathrm{A}}_g(TX, \nabla^{TX})$ only depends on $g\in G$ and $R^{TZ}$, we can denote
it by $\widehat{\mathrm{A}}_g(R^{TX})$.
Let $R^{TZ}_T$ be the curvature of $\nabla^{TZ,T}$. We define 
(cf. \cite[(4.19)]{liu2017functoriality}),
\begin{equation}
\gamma_A(T) := -\left.\frac{\partial}{\partial s}\right|_{s=0} \widehat{\mathrm{A}}_g\left( R^{TZ}_T + s\frac{\partial\nabla^{TZ,T}}{\partial T} \right).
\end{equation}
%By Chern-Weil theory, we get 
%\begin{equation}
%\frac{\partial}{\partial T}\widetilde{\widehat{\mathrm{A}}}_g(TZ,\nabla^{TZ},\nabla^{TZ,T}) = -\gamma_A(T).
%\end{equation}
By \cite[Proposition 4.5]{liu2017functoriality}, when $T\to+\infty$, $\gamma_A(T)=\mathrm{O}(T^{-2})$, and modulo exact forms on $W^g$,
\begin{equation}
\widetilde{\widehat{\mathrm{A}}}_g(TZ,\nabla^{TZ},{^0\nabla^{TZ}}) = - \int_1^{+\infty}\gamma_A(T)d T.
\end{equation}

The following two theorems are proved in \cite{liu2017functoriality}.
\begin{theorem}\label{thmintermediate1}\cite[Theorem 4.3]{liu2017functoriality}
  \begin{enumerate}
    \item For any $u>0$,
    \begin{equation}
      \lim_{T\to+\infty}\gamma^u(T,u) = \gamma_1(u).
    \end{equation}
    \item For fixed $0<u_1<u_2<+\infty$, there exists $C>0$ such that, for $u\in [u_1,u_2],T\geqslant 1$, we have 
    \begin{equation}
      |\gamma^u(T,u)|\leqslant C.
    \end{equation}
  \end{enumerate}
%  \begin{proof}
%    The proof is given by \cite[\S 5.3]{liu2017functoriality}.
%  \end{proof}
\end{theorem}

\begin{theorem}\label{thmintermediate4}\cite[Theorem 4.6]{liu2017functoriality}
	\begin{enumerate}
		\item For fixed $0<u_1<u_2<+\infty$, there exist $\delta\in (0,1], C>0$ and $T_0\geqslant 1$, such that for any $u\in [u_1,u_2], T\geqslant T_0$, we have 
		\begin{equation}
		|\gamma^T(T,u)|\leqslant\frac{C}{T^{1+\delta}}.
		\end{equation}
		\item For any $T\geqslant 0$, we have 
		\begin{equation}
		\lim_{\varepsilon\to 0}\varepsilon^{-1}\gamma^T(T\varepsilon^{-1},\varepsilon) = \int_{Y^g}\widehat{\mathrm{A}}_g(TY,\nabla^{TY})\wedge\gamma_0(T).
		\end{equation}
		\item There exists $C>0$, such that for $\varepsilon\in (0,1]$, $\varepsilon\leqslant T\leqslant 1$,
		\begin{equation}
		\varepsilon^{-1}\left| \gamma^T(T\varepsilon^{-1},\varepsilon) + \int_{Z^g}\gamma_A(T\varepsilon^{-1}) \right| \leqslant C.
		\end{equation}
	\end{enumerate}
%	\begin{proof}
%		All proofs can be found in \cite[\S 6,7,8,9]{liu2017functoriality}
%	\end{proof}
\end{theorem}

Compared with \cite[Theorem 4.4]{liu2017functoriality}, by (\ref{eq:3.30}), (\ref{eq:3.33}) and the equivariant version
of \cite[Theorem 9.19]{berline1992heat}, we have
\begin{theorem}\label{thmintermediate3}
	For $T\geqslant 1$, 
	\begin{equation}
	\lim_{u\to+\infty} \gamma^T(T,u) = \widetilde{\psi}_S\mathrm{Tr}_s\left[g\frac{\partial \nabla^{\ker {D_{Z,T}^{\mathcal{E}}}}}{\partial T} \exp(-\nabla^{\ker {D_{Z,T}^{\mathcal{E}}},2}) \right].
	\end{equation}
\end{theorem}

The last section will be devoted to prove the following  theorem.
Comparing with \cite[Theorem 4.3 (iii)]{liu2017functoriality},
\begin{theorem}\label{thmintermediate2}
  We have the following identity:
  \begin{equation}
    \lim_{T\to+\infty}\int_1^{+\infty} \gamma^u(T,u)d u =  \int_1^{+\infty}\gamma_1(u)d u - \sum_{r=2}^\infty\widetilde{\eta}_g(\mathscr{E}_r,\mathscr{E}_{r+1},\nabla^r,\nabla^{r+1}).
  \end{equation}
\end{theorem}

By Theorems \ref{thmintermediate1} - \ref{thmintermediate3}, we can calculate $I^3_i$ in (\ref{sumofI3i}) to prove the main result Theorem \ref{thmmaintheorem}. 
From the same approach in \cite[\S 4.3]{liu2017functoriality},
we have
\begin{align}
I_1^3=-\widetilde{\eta}_g(\underline{\pi_Y},\underline{\mathcal{E}_Y\otimes\ker D_X^{\mathcal{E}_X}})- \sum_{r=2}^\infty\widetilde{\eta}_g(\mathscr{E}_r,\mathscr{E}_{r+1},\nabla^r,\nabla^{r+1}),
\end{align}
\begin{align}
I_3^3=\widetilde{\eta}_g(\underline{\pi_Z},\underline{\mathcal{E}}),
\end{align}
and
\begin{align}
I_4^3=-\int_{Y^g} \widehat{\mathrm{A}}_g(TY,\nabla^{TY})\widetilde{\eta}_g(\underline{\pi_X^g},\underline{\mathcal{E}_X}) 
+ \int_{Z^g}\widetilde{\widehat{\mathrm{A}}}_g(TZ,\nabla^{TZ},{^0\nabla^{TZ}})\mathrm{ch}_g(\mathcal{E}/\mathcal{S},\nabla^{\mathcal{E}}).
\end{align}

%This part follows the same as \cite[\S 4.3]{liu2017functoriality}. 
%By (\ref{sumofI3i}), in order to obtain our main result Theorem \ref{thmmaintheorem}, we only need to calculate $I_2^3$.

%\subsection{Calculation of $I_2^3$}
By Theorem \ref{thmintermediate3} and (\ref{nablainfinity}), 
\begin{equation}
	\begin{aligned}
		I_2^3 & = -\lim_{A\to+\infty}\lim_{u\to+\infty}\int_1^A\gamma^T(T,u)d T \\
		& = - \lim_{A\to+\infty}\int_1^A\psi_S\widetilde{\mathrm{Tr}}\left[ g \frac{\partial \nabla^{\ker {D_{Z,T}^{\mathcal{E}}}}}{\partial T}\exp(-\nabla^{\ker {D_{Z,T}^{\mathcal{E}}},2})\right]d  T \\
		& = - \lim_{A\to+\infty} \widetilde{\mathrm{ch}}_g(\ker {D_{Z}^{\mathcal{E}}},\nabla^{\ker {D_{Z,A}^{\mathcal{E}}}},\nabla^{\ker {D_{Z}^{\mathcal{E}}}})\\
		& = -\widetilde{\mathrm{ch}}_g(\ker {D_{Z}^{\mathcal{E}}},\nabla^\infty,\nabla^{\ker {D_{Z}^{\mathcal{E}}}}).
	\end{aligned}
\end{equation}
%Since $\ker D^Z\simeq \mathscr{E}_\infty$, and by the definition of $\nabla^\infty$, we may write 
%\begin{equation}
%	I_2^3 = - \widetilde{\eta}_g(\mathscr{E}_\infty,\ker D^Z,\nabla^\infty,\nabla^{\ker D^Z}).
%\end{equation}
%The proof of Theorem \ref{maintheorem} is completed.

\section{The Proof of Theorems \ref{thmintermediate2}}
The purpose of this section is to prove the Theorems   \ref{thmintermediate2}. In Section 5.1, we analyze the resolvents of Dirac operators $D_{Z,T}^{\mathcal{E}}$ and $D_r$ to establish the relations. In Section 5.2, we follow what Ma have done in \cite{ma2000formes} for the functoriality of holomorphic analytic torsions to give the proof. 

\subsection{Limits of resolvent}
%We have the following spaces on which $B_r, r\geqslant 0$ act on.
%\begin{definition}
  For $v\in V, b\in S$, we set 
  \begin{equation}
    \begin{aligned}
      \mathbb{E}_v &: = C^\infty (X_v,\pi_Z^*\Lambda(T^*S)\widehat{\otimes}\mathcal{E}_X), \quad
      &\mathbb{E}_{0,b} : = C^\infty (Z_b,\pi_Z^*\Lambda(T^*S)\widehat{\otimes}\mathcal{E}_X), \\
      \mathbb{E}_{1,b} &: = C^\infty (Y_b,\pi_Y^*\Lambda(T^*S)\widehat{\otimes}\mathcal{E}_Y\otimes\ker D_X^{\mathcal{E}_X}), \quad
      &\mathbb{E}_r : = C^\infty(S, \Lambda(T^*S)\widehat{\otimes}\mathscr{E}_r),\quad r\geqslant 2.
    \end{aligned}
  \end{equation}
 % These are spaces on which Bismut superconnections $B$, $B_0$, $B_1$ and $B_r$ act. 
  For $\mu\in\mathbb{R}$, we could make use of geometric structures to define Sobolev spaces $\mathbb{E}^\mu_v$, $\mathbb{E}^\mu_{0,b}$, $\mathbb{E}^\mu_{1,b}$, $\mathbb{E}^\mu_r$, $r\geqslant 2$ of order $\mu$ respectively.
  We shall denote by $\|\cdot\|_{X,\mu}$, $\|\cdot\|_\mu$, $\|\cdot\|_{Y,\mu}$ and $\|\cdot\|_{r,\mu}$ the corresponding Sobolev norms. We still use the notations $p^T:\mathbb{E}_0\to\Omega^\bullet(S)\widehat{\otimes}\ker D_{Z,T}^{\mathcal{E}}$ and $p_r:\mathbb{E}_0\to\mathbb{E}_r$, $r\geqslant 1$ as the orthogonal projections. We have the projections $p^{T,\bot} := 1-p^T$ and $p_r^\bot := 1-p_r$, $r\geqslant 1$.
%\end{definition}

Taking $c>0$, for $r\geqslant 2$, let 
\begin{equation}
  U_r:=\left\{ \lambda\in\mathbb{C} : \mathrm{inf}_{\mu\in\mathrm{Sp}(D_r)} |\lambda-\mu|\geqslant c \right\}.
\end{equation}
The proof of the following theorem is almost the same as \cite[Theorem 6.2]{berthomieu1994quillen}.
\begin{theorem}
  For $r\geqslant 2$, $\lambda\in U_r$, there exist linear maps
  \begin{equation}\label{varphirlambda}
    \varphi_{r,\lambda} : \mathbb{E}_0 \to \mathbb{E}_0^{r+1},
  \end{equation}
  such that for $s\in \mathbb{E}_0$, we write $\varphi_{r,\lambda}(s) = (s_0,\cdots,s_r)$, which satisfies
  \begin{equation}\label{DXsr=s}
    \begin{aligned}
      D_X^{\mathcal{E}_X}s_0 & =0, \\
      D_Hs_0 + D_X^{\mathcal{E}_X}s_1 & = 0, \\
      &\vdots \\
      D_X^{\mathcal{E}_X}s_{r-1}+D_Hs_{r-2}+Cs_{r-3} & = 0, \\
      -D_X^{\mathcal{E}_X}s_r-D_Hs_{r-1}-Cs_{r-2} +\lambda s_0 & = s.
    \end{aligned}
  \end{equation}
  And for any $\mu\in \mathbb{R}$, $\varphi_{r,\lambda}$ can be extended to a bounded linear map from $\mathbb{E}_0^\mu$ to $(\mathbb{E}_0^\mu)^{r+1}$.
  Moreover, we have $s_0\in \mathscr{E}_r$, which can be given by 
  \begin{equation}\label{s0=prs}
    s_0 = (\lambda-D_r)^{-1}p_rs.
  \end{equation}
\end{theorem}

Let $\alpha_T$, $T\in[1,+\infty]$ be a family of tensors or differential operators. We denote by $a^{(i)}$ the derivative of $\alpha_T-\alpha_\infty$ of order $i$. If for any $p\in\mathbb{N}$, there exists $C>0$ such that when $T\geqslant 1$, $\sup\|a^{(i)}\|\leqslant C/T^k, i=0,1,\cdots,p$, we write
%\begin{equation}
 $ \alpha_T = \alpha_\infty + \mathrm{O}\left(\frac{1}{T^k}\right).$
%\end{equation}

We take $c_1,c_2$ such that 
\begin{equation}\label{c1c2definition}
  \bigcup_{r\geqslant 2} \mathrm{Sp}(D^{2,>0}_r) \subset (c_1,c_2),\quad (0,2c_1)\bigcap\bigcup_{b\in S}\mathrm{Sp}({D_{Y_b}^{\mathcal{E}_Y\otimes\ker D_X^{\mathcal{E}_X}}}) = \emptyset.
\end{equation}
Set 
\begin{equation}
  U_0:=\left\{ \lambda\in\mathbb{C}: \frac{\sqrt{c_1}}{2} \leqslant |\lambda| \leqslant\sqrt{c_1}, \text{ or } \sqrt{c_1}\leqslant |\lambda| \leqslant 2\sqrt{c_2} \right\}.
\end{equation}

Comparing with \cite[Theorem 6.5]{berthomieu1994quillen}, by Assumption \ref{assumptions}, we have the following relation of resolvents of ${D_{Z,T}^{\mathcal{E}}}$ and $D_r$. The proof of this theorem is the 
same as that of \cite[Theorem 6.5]{berthomieu1994quillen}.
We write it here for the completeness.
\begin{theorem}\label{thmATZtoAr}
	Given $r\geqslant 2$, for $\lambda\in U_0$, $s\in\mathbb{E}_0^0$, there exists $C\in\mathbb{R}$, such that when $T\to+\infty$,
	\begin{equation}\label{ATZtoAr}
		\|(\lambda - T^{r-1}{D_{Z,T}^{\mathcal{E}}})^{-1}s - p_r(\lambda-D_r)^{-1}p_r s\|\leqslant \frac{C}{T}\|s\|. 
	\end{equation}
	\begin{proof}
		Set  
		\begin{equation}\label{MrT}
			\begin{aligned}
				M_{r,T} : \mathbb{E}_0^{r+1} & \to \mathbb{E}_0, \\
				(s_0,\cdots,s_r) & \mapsto s_0 + \frac{s_1}{T} + \cdots + \frac{s_r}{T^r}.
			\end{aligned}
		\end{equation}
		Define
		\begin{equation}
			N_{r,T} = M_{r,T}\circ \varphi_{r,\lambda},
		\end{equation}
		where $\varphi_{r,\lambda}$ is defined in (\ref{varphirlambda}). By (\ref{DXsr=s}), we get  
		\begin{multline}
			%\begin{aligned}
				(\lambda - T^{r-1}{D_{Z,T}^{\mathcal{E}}}) N_{r,T} s  =(\lambda - T^rD_X^{\mathcal{E}_X} - T^{r-1}D_H - T^{r-2}C)\left(s_0+\frac{s_1}{T} + \cdots + \frac{s_r}{T^r}\right) \\
%				=\lambda s_0 + \lambda\left(\frac{s_1}{T} + \cdots + \frac{s_r}{T^r}\right) - \cancel{T^rD_X^{\mathcal{E}_X}s_0} - \cancel{T^{r-1}D_X^{\mathcal{E}_X}s_1} - \cdots - D_X^{\mathcal{E}_X}s_r \\
%				- \cancel{T^{r-1}D_Hs_0} - \cancel{T^{r-2}D_Hs_1} - \cdots - \frac{1}{T}D_H s_r 
%				- \cancel{T^{r-2}Cs_0} - \cancel{T^{r-3}Cs_1} - \cdots - \frac{1}{T^2}Cs_r \\
				= s + \lambda \left(\frac{s_1}{T}+\cdots+\frac{s_r}{T^r}\right) - \frac{1}{T}D_Hs_r - \frac{1}{T^2}Cs_r.
			%\end{aligned}
		\end{multline}
		Hence
		\begin{equation}
			(\lambda-T^{r-1}{D_{Z,T}^{\mathcal{E}}})^{-1}s = N_{r,T}s + (\lambda-T^{r-1}{D_{Z,T}^{\mathcal{E}}})^{-1}\left( -\lambda \left(\frac{s_1}{T}+\cdots+\frac{s_r}{T^r}\right) + \frac{1}{T}D_Hs_r + \frac{1}{T^2}Cs_r \right).
		\end{equation}
		Since $\|(\lambda - T^{r-1}{D_{Z,T}^{\mathcal{E}}})^{-1}\|$ is uniformly 
		bounded, by (\ref{s0=prs}), (\ref{MrT}), we obtain 
		(\ref{ATZtoAr}).
	\end{proof}
\end{theorem}

\begin{corollary}\label{corkerDZEinfty}
There exist $C>0$, $T_0\geq 1$, such that for any 
	$T\geq T_0$, $s\in \mathscr{E}_0$,
	\begin{align}\label{eq:5.13}
	\| p^Ts-p_{\infty}s\|\leq \frac{C}{T}\|s\|.
	\end{align}
  Moreover, for $T\geq 1$, 
  \begin{equation}
    \ker {D_{Z,T}^{\mathcal{E}}}\simeq\mathscr{E}_\infty.
  \end{equation}
  \begin{proof}
    Set 
    \begin{equation}\label{eq:5.15}
      \begin{aligned}
        P_{r,T} & := \frac{1}{2\pi\sqrt{-1}}\int_{\{\lambda\in\mathbb{C}: |\lambda|<\sqrt{c_1}\}}(\lambda-T^{r-1}{D_{Z,T}^{\mathcal{E}}})^{-1}d \lambda \\
        & = \frac{1}{2\pi\sqrt{-1}}\int_{\{\lambda\in\mathbb{C}:|\lambda|<\frac{\sqrt{c_1}}{T^{r-1}}\}}(\lambda-{D_{Z,T}^{\mathcal{E}}})^{-1}d \lambda.
      \end{aligned}
    \end{equation}
    By Lemma \ref{lemEr=kerDr-1} and (\ref{c1c2definition}), 
    \begin{equation}
      P_r := \frac{1}{2\pi\sqrt{-1}}\int_{\{\lambda\in\mathbb{C}: |\lambda|<\sqrt{c_1}\}}p_r(\lambda-D_r)^{-1}p_rd \lambda = p_{r+1}.
    \end{equation}
    Recall that $r_0$ is the index from which $\mathscr{E}_r$ converges. Hence $P_{r_0-1} = P_{r_0}=\cdots = P_\infty$.

    By Theorem \ref{thmATZtoAr}, when $T\gg 1$,
    \begin{equation}\label{PTn-Pnestimate}
      \|P_{r_0-1,T}s - P_{r_0-1}s\| \leqslant\frac{C}{T}\|s\|.
    \end{equation}
    Note that when $n>r_0-1$, $P_{n,T} = P_{r_0-1,T}$. As $n\to+\infty$, by (\ref{eq:5.15}), $P_{r_0-1,T} = p^T$. 
    %By $\ker D^Z_T\subset P_{r_0,T}$, we have $P_{r_0,T}=P_{\ker D^Z_T}$. 
    So we get (\ref{eq:5.13}).
    
    By Assumption \ref{assumptions}, $\dim\ker {D_{Z,T}^{\mathcal{E}}}$ is independent of $T$.
    According to (\ref{PTn-Pnestimate}), $\dim\mathrm{im}P_{r_0} = \dim\ker {D_{Z,T}^{\mathcal{E}}}$. So $\dim\ker {D_{Z,\infty}^{\mathcal{E}}} = \dim\mathscr{E}_\infty$.
    Under Assumption \ref{assumptions},  $\ker {D_{Z,T}^{\mathcal{E}}}\cong\ker {D_{Z}^{\mathcal{E}}}$ are isomorphic vector bundles. Hence $\ker {D_{Z}^{\mathcal{E}}}\cong\mathscr{E}_\infty$.
    
    The proof of Corollary \ref{corkerDZEinfty} is completed.
  \end{proof}
\end{corollary}

Theorem \ref{thmATZtoAr} is essential for our analysis. It tells us that all eigenvalues of ${D_{Z,T}^{\mathcal{E}}}$ which satisfy $\mathrm{O}\left(1/T^{r-1}\right)$ obey $\frac{1}{T^{r-1}}\left(\mathrm{Sp}(D_r) + \mathrm{O}(1/T)\right)$. We depict the contours $\delta_0, \Delta_0$ in Figure \ref{contour2}:
\begin{figure}[!htp]\centering\begin{tikzpicture}[>=stealth,thick]
  \draw[->] (-1. 2,0) --  (4. 5,0) node[right] {$x$};
  \draw[->] (0,-1. 5) -- (0,1. 5) node[above right] {$y$};

  \draw[->] (1. 5,1) -- (1. 5,0. 3) node[right] {\footnotesize $c_1$};
  \draw (1. 5,0. 3) -- (1. 5,-1); 
  \draw[->] (1. 5,-1) -- (2,-1);
  \draw (2,-1) -- (4,-1); 
  \draw[->] (4,-1) -- (4,0. 3) node[right] {\footnotesize $c_2$};
  \draw (4,0. 3) -- (4,1); 
  \draw[->] (4,1) -- (2,1) node[right=10pt,fill=white] {$\Delta_0$};
  \draw (2,1) -- (1. 5,1); 

  \draw[dashed] (1. 5,1) -- (0,1) node[left] {$1$};
  \draw[dashed] (1. 5,-1) -- (0,-1) node[left] {$-1$};

\node at (0. 3,0. 3) {$\mathcal{U}_1$};
\node at (3. 3,0. 3) {$\mathcal{U}_2$};

\draw[->] (0. 75,0) arc (0:180:0. 75) node[below left] {\footnotesize $\frac{c_1}{4}$};
\draw (-0. 75,0) arc (180:360:0. 75);
\node[fill=white] at (0. 5,-0. 5) {\footnotesize $\delta_0$};
\end{tikzpicture} 
\caption{Contours $\delta_0,\Delta_0$}\label{contour2}
\end{figure}
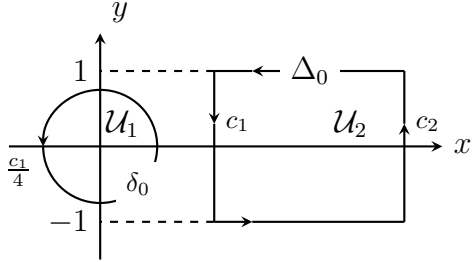

By (\ref{c1c2definition}), we know that when $T\gg 1$, 
\begin{equation}\label{spectradecomposition}
  \mathrm{Sp}({D_{Z,T}^{\mathcal{E},2}}) \bigcap [0,2c_1] \subset \frac{\mathcal{U}_1}{T^{2(r_0-1)}}\bigcup\bigcup_{r=2}^{r_0} \frac{\mathcal{U}_2}{T^{2(r-1)}}.
\end{equation}

Comparing with \cite[Proposition 6.12]{bunke2004index}, we show that Theorem \ref{thmmaintheorem} really extends Dai's adiabatic limit formula to the family case. 
\begin{proposition}\label{prop:5.04}
  If $S$ is a point, set 
  \begin{equation}
    A_r := \left\{ \lambda\in\mathrm{Sp}({D_{Z,T}^{\mathcal{E}}}):\lambda=\mathrm{O}\left(\frac{1}{T^{r-1}}\right) \right\}.
  \end{equation}
 Then when $T\to+\infty$,
  \begin{equation}
    \widetilde{\eta}_e(\mathscr{E}_r,\mathscr{E}_{r+1},\nabla^r,\nabla^{r+1}) = \sum_{\lambda\in A_r/A_{r+1}}\mathrm{sgn}(\lambda).
  \end{equation}

\end{proposition}
  \begin{proof}
	By Definition \ref{defnfiniteetafrom}, when $S$ is a point, 
	\begin{equation}
	\widetilde{\eta}_e(\mathscr{E}_r,\mathscr{E}_{r+1},\nabla^r,\nabla^{r+1}) = \int_0^{+\infty}\widetilde{\psi}_{\{\mathrm{pt}\}}\mathrm{Tr}_s\left[\frac{D_r}{2\sqrt{t}}\exp(-tD^2_r)\right]d t.
	\end{equation}
	By Gauss integral, as 
	\begin{equation*}
	\frac{1}{\sqrt{\pi}} \int_0^{+\infty}\lambda\mathrm{e}^{-t\lambda^2}\frac{d t}{2\sqrt{t}}=\mathrm{sgn}(\lambda),
	\end{equation*}
	we have 
	\begin{equation*}
	\widetilde{\eta}_e(\mathscr{E}_r,\mathscr{E}_{r+1},\nabla^r,\nabla^{r+1}) = \sum_{\lambda\in\mathrm{Sp}(D_r)}\mathrm{sgn}(\lambda).
	\end{equation*}
	According to Theorem \ref{thmATZtoAr}, we see that $\widetilde{\eta}_e(\mathscr{E}_r,\mathscr{E}_{r+1},\nabla^r,\nabla^{r+1}) = \sum_{\lambda\in A_r/A_{r+1}}\mathrm{sgn}(\lambda)$.
\end{proof}

\subsection{Proof of Theorem \ref{thmintermediate2}}
We start from the definitions of $\gamma^u(T,u),\gamma^T(T,u)$ and $\gamma_r(u)$. Set 
\begin{equation}\label{calBs}
	\begin{aligned}
		\mathcal{B}_T & := B_T^2 + d  u\wedge\delta_{u^2}^{-1}\frac{\partial B_{u^2,T}}{\partial u}\delta_{u^2}, \\
		\mathcal{B}_{u,T} & := B^2_{u^2,T} + d  u\wedge \frac{\partial B_{u^2,T}}{\partial u}, \\
		\mathcal{B}_r & := B_r^2 + d  u\wedge\delta_{u^2}^{-1}\frac{\partial B_{r,u^2}}{\partial u}\delta_{u^2}\quad r\geqslant 1,
	\end{aligned}
\end{equation}
where $B_{r,u^2}=u\, \delta_{u^2}\circ B_r\circ \delta_{u^2}^{-1}$.
Then
\begin{equation}\label{gammacalBs}
	\begin{aligned}
		\gamma^u(T,u) & = \left\{ \psi_S\widetilde{\mathrm{Tr}}[g\exp(-\mathcal{B}_{u^2,T})] \right\}^{d  u} = \left\{ u^{-2}\psi_S\delta_{u^2}\widetilde{\mathrm{Tr}}[g\exp(-u^2\mathcal{B}_T)] \right\}^{d  u}, \\
		\gamma_r(u) & = \left\{ u^{-2}\psi_S\delta_{u^2}\widetilde{\mathrm{Tr}}[g\exp(-u^2\mathcal{B}_r)] \right\}^{d  u},\quad r\geqslant 1. 
	\end{aligned}
\end{equation}
Set
\begin{equation}\label{BruT}
	\mathcal{B}_{r,u,T}:=\mathcal{B}_{T^{r-1}u,T}. 
\end{equation}

The proof of the following theorem is the same as \cite[Theorem 9.2]{bismut1997holomorphic} (see also \cite[Lemma 5.8]{liu2017functoriality}).
%We have the following relations of spectrum of superconnections $\mathcal{B}_{r,1,T}$ and $\mathcal{B}_r$: 
\begin{theorem}\label{spectrumofB}
	For $u>0$, $T\geq 1$, we have
	\begin{equation}
		\begin{aligned}
			\mathrm{Sp}(\mathcal{B}_{u,T}) & = \mathrm{Sp}(u^2\mathcal{B}_{1,T})= \mathrm{Sp}(u^2{D_{Z,T}^{\mathcal{E},2}}),\quad \mathrm{Sp}(\mathcal{B}_{r,u,T}) =  \mathrm{Sp}(T^{2(r-1)}u^2{D_{Z,T}^{\mathcal{E},2}}), \\
			\mathrm{Sp}(\mathcal{B}_1) & = \mathrm{Sp}({D_Y^{\mathcal{E}_Y\otimes D_X^{\mathcal{E}_X},2}}), \quad
			\mathrm{Sp}(\mathcal{B}_r)  = \mathrm{Sp}(D^2_r),\quad r\geqslant 2. 
		\end{aligned}
	\end{equation}
%	\begin{proof}
%		
%	\end{proof}
\end{theorem}

Set 
\begin{equation}\label{FrTGrT}
	\begin{aligned}
		F_{r,u,T} & := u^{-2}\psi_S\delta_{u^2}\widetilde{\mathrm{Tr}}\left[ \int_{\Delta_0} \mathrm{e}^{-u^2\lambda}(\lambda-\mathcal{B}_{r,1,T})^{-1}d \lambda \right].\quad r\geq 2; \\
		F_{r,u,\infty} & := u^{-2}\psi_S\delta_{u^2}\widetilde{\mathrm{Tr}}\left[ \int_{\Delta_0} \mathrm{e}^{-u^2\lambda}(\lambda-\mathcal{B}_r)^{-1}d \lambda \right], \quad r\geq 2; \\
		G_{r,u,T} & := u^{-2}\psi_S\delta_{u^2}\widetilde{\mathrm{Tr}}\left[ \int_{\delta_0} \mathrm{e}^{-u^2\lambda}(\lambda-\mathcal{B}_{r,1,T})^{-1}d \lambda \right],\quad r\geq 1;\\
		G_{r,u,\infty} & := u^{-2}\psi_S\delta_{u^2}\widetilde{\mathrm{Tr}}\left[ \int_{\delta_0} \mathrm{e}^{-u^2\lambda}(\lambda-\mathcal{B}_r)^{-1}d \lambda \right],\quad r\geq 1. 
	\end{aligned}
\end{equation}
Note that $\mathcal{B}_{1,1,T}=\mathcal{B}_{1,T}=\mathcal{B}_T$.
Set
\begin{align}
F_{1,u,T}=\psi_S\widetilde{\mathrm{Tr}}[g\exp(-\mathcal{B}_{u,T})]-G_{1,u,T},\quad F_{1,u,\infty}=\psi_S\widetilde{\mathrm{Tr}}[g\exp(-\mathcal{B}_{1,u^2})]-G_{1,u,\infty}.
\end{align}
By  (\ref{gammacalBs}), for $T\gg 1$, 
\begin{equation}\label{F+G=gamma}
	\begin{aligned}
	\gamma^u(T,u) & = 	\{F_{1,u,T}\}^{d  u} + \{G_{1,u,T}\}^{d  u} \\
	\gamma_r(u) & = \left\{ F_{r,u,\infty} \right\}^{d  u}+ \left\{ G_{r,u,\infty} \right\}^{d  u},\quad  r\geqslant 2. 
	\end{aligned}
\end{equation}
When $T\gg 1$, we have
\begin{equation}\label{G1uTdecomposition}
	\begin{aligned}
		G_{1,u,T} = & \sum_{r=2}^{r_0} \psi_S\delta_{u^2}\widetilde{\mathrm{Tr}}\left[ \int_{\frac{\Delta_0}{T^{2(r-1)}}} \mathrm{e}^{-u^2\lambda}(\lambda-\mathcal{B}_{1,T})d \lambda \right] \\
		& +  \psi_S\delta_{u^2}\widetilde{\mathrm{Tr}}\left[ \int_{\frac{\delta_0}{T^{2(r_0-1)}}} \mathrm{e}^{-u^2\lambda}(\lambda-\mathcal{B}_{1,T})d \lambda \right] \\
		= & \sum_{r=2}^{r_0} F_{r,T^{1-r}u,T} + G_{r_0,T^{1-r_0}u,T}. 
	\end{aligned}
\end{equation}

		The proof of the following lemma is same with that in \cite[(2.98) and (2.105)]{ma2002functoriality} and \cite[\S 2.e]{ma2000formes}. Note that in our situation, the Dirac operator cannot be decomposed into the sum of two nilpotent operators. But the term $p_{r,T}$ in \cite[\S 2.e]{ma2000formes} is the same as $p_r$ in our case, which is independent of $T$. This is also the case in \cite{ma2002functoriality}. Hence we can prove this lemma for general Dirac operators.
\begin{lemma}\label{FruTFruinftyapproximation}
	(1) There exist $\delta,c,C,T_0>0$ such that for any $u\geqslant 1,T\geqslant T_0,r\geqslant 1$,
	\begin{equation}\label{eq:5.30}
		\left|F_{r,u,T}-F_{r,u,\infty}\right|\leqslant\frac{C}{T^\delta}\mathrm{e}^{-cu},\quad \left|G_{r,u,T}-G_{r,u,\infty}\right|\leqslant\frac{C}{T^\delta}.
	\end{equation}
	
	(2) There exist $C,\delta>0$, such that for any $u\in\mathbb{C}, |u|\leqslant 1, T\geqslant T_0$, 
	\begin{equation}\label{eq:5.40}
	|F_{r,u,T} - F_{r,u,\infty} |\leqslant \frac{C}{T^\delta},\quad |G_{r,u,T} - G_{r,u,\infty} |\leqslant \frac{C}{T^\delta}.
	\end{equation}
%	\begin{proof}
%
%	\end{proof}
\end{lemma}

%As
%\begin{equation}\label{Fduu-1Fds}
%	\{F_{r,u,T}\}^{d  u} = u^{-1}\cdot\{F_{r,su,T}\}^{d  s}|_{s=1}, 
%\end{equation}
Set
\begin{equation}\label{frTgrT}
	\begin{aligned}
		f_{r,u,T} & = \{F_{r,su,T}\}^{d  s}|_{s=1}=u\{F_{r,u,T}\}^{d  u}, \\
		f_{r,u,\infty} & = \{F_{r,su,\infty}\}^{d  s}|_{s=1}=u\{F_{r,u,\infty}\}^{d  u}, \\
		g_{r,u,T} & = \{G_{r,su,T}\}^{d  s}|_{s=1}=u\{G_{r,u,T}\}^{d  u}, \\
		g_{r,u,\infty} & = \{G_{r,su,\infty}\}^{d  s}|_{s=1}=u\{G_{r,u,\infty}\}^{d  u}. 
	\end{aligned}
\end{equation}
By (\ref{F+G=gamma}) and (\ref{frTgrT}), 
\begin{equation}\label{f+g=gamma}
	\begin{aligned}
		f_{1,u,T} + g_{1,u,T} & = u\gamma^u(T,u), \\
		f_{r,u,\infty} + g_{r,u,\infty} & = u\gamma_r(u),\quad r\geq 2. 
	\end{aligned}
\end{equation}
By (\ref{eq:5.30}), when $r\geqslant 1$,
\begin{equation}\label{fruTfruinftyapproximation}
	\left|f_{r,u,T} - f_{r,u,\infty}\right|\leqslant\frac{Cu}{T^\delta }\mathrm{e}^{-cu}, \quad \left|g_{r,u,T} - g_{r,u,\infty}\right|\leqslant\frac{Cu}{T^\delta}. 
\end{equation}
Note that when $r=1$, the corresponding result is given by Theorem \ref{thmintermediate1}(1).

By (\ref{G1uTdecomposition}), (\ref{frTgrT})-(\ref{fruTfruinftyapproximation}) and the dominated
convergence theorem,
\begin{multline}\label{eq:5.33}
\lim_{T\to+\infty} \int_1^{+\infty} \gamma^u(T,u)d  u
=\lim_{T\to+\infty} \int_1^{+\infty} \left\{ \gamma^u(T,u) - \frac{g_{1,u,T}}{u} \right\} d  u 
+\lim_{T\to+\infty} \int_1^{+\infty} g_{1,u,T} \frac{d  u}{u}
\\
= \int_1^{+\infty} \left\{ \gamma_1(u) - \frac{g_{1,u,\infty}}{u} \right\} d  u+\lim_{T\to+\infty} \int_1^{+\infty} g_{1,u,T} \frac{d  u}{u}
\\
=\int_1^{+\infty} \left\{ \gamma_1(u) - \frac{g_{1,u,\infty}}{u} \right\} d  u+\lim_{T\to+\infty}\sum_{r=2}^{r_0}   \int_1^{+\infty} f_{r,T^{1-r}u,T} \frac{d  u}{u} + \lim_{T\to+\infty} \int_1^{+\infty} g_{r_0,T^{1-r_0}u,T} \frac{d  u}{u}
\\
=\int_1^{+\infty} \gamma_1(u)  d  u+
\lim_{T\to+\infty}\sum_{r=2}^{r_0}  \int_{T^{1-r}}^{+\infty} f_{r,u,T} \frac{d  u}{u} + \lim_{T\to+\infty} \int_{T^{1-r_0}}^{+\infty} g_{r_0,u,T} \frac{d  u}{u} - \int_1^{+\infty} g_{1,u,\infty} \frac{d  u}{u}
\\
=\int_1^{+\infty} \gamma_1(u)  d  u+\lim_{T\to+\infty}Q_{1,T}+\lim_{T\to+\infty}Q_{2,T}
- \int_1^{+\infty} g_{1,u,\infty} \frac{d  u}{u},
\end{multline}
where
\begin{equation}\label{Q1TQ2T}
	\begin{aligned}
		Q_{1,T} := \sum_{r=2}^{r_0} & \int_1^{+\infty} f_{r,u,T} \frac{d  u}{u} + \int_1^{+\infty} g_{r_0,u,T} \frac{d  u}{u}, \\
		Q_{2,T} := \sum_{r=2}^{r_0} & \int^1_{T^{1-r}} f_{r,u,T} \frac{d  u}{u} + \int^1_{T^{1-r_0}} g_{r_0,u,T} \frac{d  u}{u}. 
	\end{aligned}
\end{equation}
By (\ref{fruTfruinftyapproximation}) and the dominated convergence theorem, when $T\to+\infty$, 
\begin{equation}\label{eq:5.35}
	Q_{1,T} \to Q_{1,\infty} = \sum_{r=2}^{r_0} \int_1^{+\infty} f_{r,u,\infty} \frac{d  u}{u} + \int_1^{+\infty} g_{r_0,u,\infty} \frac{d  u}{u}. 
\end{equation}

Then Theorem \ref{thmintermediate2} follows directly from 
(\ref{eq:4.11}), (\ref{f+g=gamma}), (\ref{eq:5.33})-(\ref{eq:5.35}) and the
following lemma.

\begin{lemma}\label{lemma:5.7}
When $T\to +\infty$,
\begin{align}\label{eq:5.36}
\lim_{T\to+\infty}Q_{2,T}
= \sum_{r=2}^{r_0} \int_0^1 \gamma_r(u)  d  u + \sum_{r=1}^{r_0-1}\int_1^{+\infty}  g_{r,u,\infty}  \frac{d  u}{u}.
\end{align}
\end{lemma}
\begin{proof}
	When $u\to 0$, by (\ref{FrTGrT}), there exists $N\in \mathbb{N}$, such that for $T\in [1,+\infty]$, $r\geq 1$, $F_{r,u,T}, G_{r,u,T}$ have asymptotic expansions
\begin{equation}
%\begin{aligned}
F_{r,u,T} = \sum_{i=-N-1}^{-1} A_{r,i,T} u^i + \mathrm{O}(1),  \quad
G_{r,u,T} = \sum_{i=-N-1}^{-1} B_{r,i,T} u^i + \mathrm{O}(1). 
%\end{aligned}
\end{equation}
For $T\in [1,+\infty]$, $r\geq 1$, set
\begin{equation}
a_{r,i,T} = \{A_{r,i-1,T}\}^{d  u},\quad b_{r,i,T} = \{B_{r,i-1,T}\}^{d  u}. 
\end{equation}
Then by (\ref{frTgrT}),
\begin{align}
f_{r,u,T} = \sum_{i=-N}^0 a_{r,i,T}u^i+\mathrm{O}(u),\quad 
g_{r,u,T} =\sum_{i=-N}^0 b_{r,i,T}u^i+\mathrm{O}(u).
\end{align}
%	By (\ref{f+g=gamma}), (\ref{fruTfruinftyapproximation}), we know that there exist $C,\delta>0$, such that for any $u\in\mathbb{C}, |u|\leqslant 1, T\geqslant T_0$, 
%\begin{equation}\label{eq:5.40}
%|f_{r,u,T} - f_{r,u,\infty} |\leqslant \frac{C}{T^\delta},\quad |g_{r,u,T} - g_{r,u,\infty} |\leqslant \frac{C}{T^\delta}.
%\end{equation}
From (\ref{eq:5.40}), when $T\to +\infty$, the functions $\{f_{r,u,T},g_{r,u,T}\}$ are uniformly bounded holomorphic functions on $\{u\in\mathbb{C}: |u|\leqslant 1\}$. Hence they have uniform expansions in the domain of $u$. By (\ref{eq:5.40}) and Cauchy formula, the coefficients of expansions of $f$ and $g$ are convergent in the sense of $\mathrm{O}(\frac{1}{T^\delta})$ when $T\to+\infty$. So  $a_{r,i,T}, b_{r,i,T}\in \Omega^{\bullet}(S)$
and depend smoothly on $T\in [1,+\infty]$. Moreover there exists
$\delta>0$, for $T\to 
+\infty$,
\begin{align}
a_{r,i,T}  = a_{r,i,\infty} + \mathrm{O}\left(\frac{1}{T^\delta}\right),\quad
b_{r,i,T}  = b_{r,i,\infty} + \mathrm{O}\left(\frac{1}{T^\delta}\right).
\end{align} 
\end{proof}
Note that by (\ref{FrTGrT}) and (\ref{frTgrT}),
		\begin{equation}
f_{r,u,\infty} + g_{r,u,\infty} =\left\{ \psi_S \widetilde{\mathrm{Tr}} \left[g\exp(-\mathcal{B}_{r, u^2})\right]\right\}^{d u}.
\end{equation}
So for $i\leq 0$, 
\begin{align}
	a_{r,i,\infty}+b_{r,i,\infty}=0.
\end{align}
		From the equivariant version of \cite[Theorem 9.7]{berline1992heat}, 
		\begin{equation}\label{fgrutuinfty}
			\lim_{u\to+\infty} (f_{r,u,T}+g_{r,u,T}) = 0.
		\end{equation}
		By the definition of $f_{r,u,T}$, $\lim_{u\to+\infty}f_{r,u,T} = 0$. Combined with (\ref{fgrutuinfty}), $\lim_{u\to+\infty}g_{r,u,T} = 0$, we have
		\begin{equation}\label{gruinftygnuT}
			g_{r,u,\infty}  = \sum_{-N}^{-1} b_{r,i,\infty}u^i.
		\end{equation}

By the definition of $c_1,c_2$ (\ref{c1c2definition}) and relation (\ref{spectradecomposition}), in the region $\mathcal{U}_1$, $0$ is the only eigenvalue of $\mathcal{B}_{r_0, u,T}$. By (\ref{FrTGrT}),
\begin{align}
G_{r_0,u,T} := u^{-2}\psi_S\delta_{u^2}\widetilde{\mathrm{Tr}}\left[ \int_{\delta_0} \mathrm{e}^{-\lambda}(\lambda-u^2\mathcal{B}_{r_0,1,T})^{-1}d \lambda \right].
\end{align}
By the same argument in the proof of \cite[Theorem 9.29]{bismut1997holomorphic}, we obtain that
$b_{r_0, i, T}=0$ for $i\geq 0$. That is,
\begin{align}
g_{r_0,u,T}=\sum_{i=-N}^{-1}b_{r_0, i, T}u^i.
\end{align}
By (\ref{G1uTdecomposition}),
\begin{align}
b_{r_0,i,T} T^{(1-r_0)i} + \sum_{r=2}^{r_0} a_{r,i,T} T^{(1-r)i} = -a_{1,i,T} \quad i<0. 
\end{align}
So we may write $Q_{2,T}$ as:
\begin{equation}\label{Q2Tcomputation1}
\begin{aligned}
Q_{2,T} = & \sum_{r=2}^{r_0} \int_{T^{1-r}}^1 \left\{ f_{r,u,T} - \sum_{i=-N}^0 a_{r,i,T}u^i \right\} \frac{d  u}{u} + \sum_{i=-N}^{-1} \frac{1}{i} \sum_{r=2}^{r_0} \left( a_{r,i,T} - a_{r,i,T}T^{(1-r)i} \right) \\
& \quad + \sum_{r=2}^{r_0} \int_{T^{1-r}}^1 a_{r,0,T} \frac{d  u}{u} + \sum_{i=-N}^{-1} \frac{1}{i} \left( b_{r_0,i,T} - b_{r_0,i,T}T^{(1-r)i} \right) \\
= & \sum_{r=2}^{r_0} \int_{T^{1-r}}^1 \left\{ f_{r,u,T} - \sum_{i=-N}^0 a_{r,i,T}u^i \right\} \frac{d  u}{u} + \sum_{i=-N}^{-1} \frac{1}{i} \left( b_{r_0,i,T} + \sum_{r=1}^{r_0}  a_{r,i,T} \right) \\
& \quad + \sum_{r=2}^{r_0} (r-1)a_{r,0,T}\log T. 
\end{aligned}
\end{equation}
So when $T\to +\infty$,
\begin{equation}
\begin{aligned}
Q_{2,T} \to Q_{2,\infty} & = \sum_{r=2}^{r_0} \int_0^1 \left\{ f_{r,u,\infty} + \sum_{i=-N}^0 b_{r,i,\infty} u^i \right\} \frac{d  u}{u} + \sum_{r=1}^{r_0-1}\sum_{i=-N}^{-1}\frac{1}{i}a_{r,i,\infty} \\
& = \sum_{r=2}^{r_0} \int_0^1 \left\{ f_{r,u,\infty} + g_{r,u,\infty} \right\} \frac{d  u}{u} - \sum_{r=1}^{r_0-1}\int_1^{+\infty} \left\{ \sum_{i=-N}^{-1} a_{r,i,\infty}u^i \right\} \frac{d  u}{u} \\
& = \sum_{r=2}^{r_0} \int_0^1 u\gamma_r(u)  \frac{d  u}{u} + \sum_{r=1}^{r_0-1}\int_1^{+\infty}  g_{r,u,\infty}  \frac{d  u}{u}.
\end{aligned}
\end{equation}

The proof of Lemma \ref{lemma:5.7} is completed.

\section*{Acknowledgements}
This paper is a condensed form of the second author's Ph.D thesis, we would like to thank Professers Xianzhe Dai and Hang Wang for helpful discussions. B.\ L.\ is partially supported by Science and Technology Commission 
of Shanghai Municipality (STCSM), grant No. 22DZ2229014, and NSFC No.11931007, No.12225105.
 
\bibliographystyle{acm}
\bibliography{Composition}

\begin{thebibliography}{10}

\bibitem{atiyah1973spectral}
{\sc Atiyah, M.~F., Patodi, V.~K., and Singer, I.~M.}
\newblock Spectral asymmetry and {R}iemannian geometry.
\newblock {\em Bull. London Math. Soc. 5\/} (1973), 229--234.

\bibitem{berline1992heat}
{\sc Berline, N., Getzler, E., and Vergne, M.}
\newblock {\em Heat Kernels and Dirac Operators: Grundlehren 298}.
\newblock Springer-Verlag, 1992.

\bibitem{berthomieu1994quillen}
{\sc Berthomieu, A., and Bismut, J.-M.}
\newblock Quillen metrics and higher analytic torsion forms.
\newblock {\em J. Reine Angew. Math. 457\/} (1994), 85--184.

\bibitem{bismut1986atiyah}
{\sc Bismut, J.-M.}
\newblock The {Atiyah-Singer} index theorem for families of {Dirac} operators:
  Two heat equation proofs.
\newblock {\em Invent. Math. 83}, 1 (1986), 91--151.

\bibitem{bismut1997holomorphic}
{\sc Bismut, J.-M.}
\newblock Holomorphic families of immersions and higher analytic torsion forms.
\newblock {\em Ast\'{e}risque}, 244 (1997), viii+275.

\bibitem{bismut1989eta}
{\sc Bismut, J.-M., and Cheeger, J.}
\newblock $\eta$-invariants and their adiabatic limits.
\newblock {\em J. Amer. Math. Soc. 2}, 1 (1989), 33--70.

\bibitem{bismut1986analysis}
{\sc Bismut, J.-M., and Freed, D.~S.}
\newblock The analysis of elliptic families. {I}. metrics and connections on
  determinant bundles.
\newblock {\em Comm. Math. Phys. 106}, 1 (1986), 159--176.

\bibitem{bismut1991complex}
{\sc Bismut, J.-M., and Lebeau, G.}
\newblock Complex immersions and {Quillen} metrics.
\newblock {\em Publ. Math. Inst. Hautes Études Sci. 74}, 1 (1991), 1--291.

\bibitem{bunke2004index}
{\sc Bunke, U., and Ma, X.}
\newblock Index and secondary index theory for flat bundles with duality.
\newblock In {\em Aspects of boundary problems in analysis and geometry}.
  Springer, 2004, pp.~265--341.

\bibitem{dai1991adiabatic}
{\sc Dai, X.}
\newblock Adiabatic limits, nonmultiplicativity of signature, and {Leray}
  spectral sequence.
\newblock {\em J. Amer. Math. Soc. 4}, 2 (1991), 265--321.

\bibitem{dai1998higher}
{\sc Dai, X., and Zhang, W.}
\newblock Higher spectral flow.
\newblock {\em J. Funct. Anal. 157}, 2 (1998), 432--469.

\bibitem{de1973varietes}
{\sc De~Rham, G.}
\newblock {\em Vari{\'e}t{\'e}s diff{\'e}rentiables: formes, courants, formes
  harmoniques}, vol.~3.
\newblock Editions Hermann, 1973.

\bibitem{liu2017functoriality}
{\sc Liu, B.}
\newblock Functoriality of equivariant eta forms.
\newblock {\em J. Noncommut. Geom. 11}, 1 (2017), 225--307.

\bibitem{liu2021equivariant}
{\sc Liu, B.}
\newblock Equivariant eta forms and equivariant differential {K}-theory.
\newblock {\em Sci. China Math. 64}, 10 (2021), 2159--2206.

\bibitem{liu2021bismut}
{\sc Liu, B.}
\newblock Bismut-{C}heeger eta form and higher spectral flow.
\newblock {\em Int. Math. Res. Not. IMRN\/} (online).
\newblock Doi: 10.1093/imrn/rnac157.

\bibitem{liumaInvent}
{\sc Liu, B., and Ma, X.}
\newblock Differential {$K$}-theory and localization formula for
  {$\eta$}-invariants.
\newblock {\em Invent. Math. 222}, 2 (2020), 545--613.

\bibitem{liuma2022}
{\sc Liu, B., and Ma, X.}
\newblock Comparison of two equivariant {$\eta$}-forms.
\newblock {\em Adv. Math. 404}, part A (2022), Paper No. 108163, 76.

\bibitem{xiaonan1999formes}
{\sc Ma, X.}
\newblock Formes de torsion analytique et familles des submersions. {I}.
\newblock {\em Bull. Soc. Math. France 127}, 4 (1999), 541--621.

\bibitem{ma2000formes}
{\sc Ma, X.}
\newblock Formes de torsion analytique et families de submersions. {II}.
\newblock {\em Asian J. Math. 4}, 3 (2000), 633--668.

\bibitem{ma2002functoriality}
{\sc Ma, X.}
\newblock Functoriality of real analytic torsion forms.
\newblock {\em Israel J. Math. 131}, 1 (2002), 1--50.

\bibitem{ma2007holomorphic}
{\sc Ma, X., and Marinescu, G.}
\newblock {\em Holomorphic {Morse} inequalities and {Bergman} kernels},
  vol.~254.
\newblock Springer Science \& Business Media, 2007.

\bibitem{quillen1985superconnections}
{\sc Quillen, D.}
\newblock Superconnections and the {Chern} character.
\newblock {\em Topology 24}, 1 (1985), 89--95.

\bibitem{Witten1985}
{\sc Witten, E.}
\newblock Global gravitational anomalies.
\newblock {\em Comm. Math. Phys. 100}, 2 (1985), 197--229.

\end{thebibliography}

\vskip 2em

{\small SCHOOL OF SCIENCE \& BIG DATA SCIENCE, ZHEJIANG UNIVERSITY OF SCIENCE AND TECHNOLOGY}

{\small HANGZHOU, ZHEJIANG, 310023, CHINA}

{\small E-mail address: 122072@zust.edu.cn}

\vspace{6pt}

{\small DEPARTMENT OF MATHEMATICS, EAST CHINA NORMAL UNIVERSITY }

{\small SHANGHAI 200062, CHINA}

{\small E-mail address: bliu@math.ecnu.edu.cn}

\end{document}